\documentclass[leqno, letterpaper, 11pt]{article}
\usepackage{amssymb}
\usepackage{amsmath}
\usepackage{amsfonts}
\usepackage{amsthm}
\usepackage[usenames]{color}
\usepackage{graphicx}
\usepackage{verbatim}
\usepackage[reftex]{theoremref}
\usepackage{bbm}

\usepackage{pgf,tikz}
\usetikzlibrary{arrows}

\newcommand{\e}{\epsilon}

\newcommand{\indicator}{{\mathbbm 1}}

\newcommand{\bZ}{{\mathbb Z}}

\newcommand{\cG}{{\mathcal G}}
\newcommand{\cH}{\mathcal H}

\newcommand{\cR}{{\mathcal R}}

\newcommand{\cO}{\mathcal O}
\newcommand{\cP}{\mathcal P}

\newcommand{\cS}{\mathcal S}

\newcommand{\size}{{\text{\tt Size}}}

\newcounter{mycount}
\newenvironment{mylist}{\begin{list}{{\rm (\roman{mycount})}}%
{\usecounter{mycount}\itemsep -2pt \topsep 0pt \labelwidth 2cm}}{\end{list}}

\newcommand{\df}{\textbf}

\newtheorem{theorem}{Theorem}[section]
\newtheorem{lemma}[theorem]{Lemma}
\newtheorem{prop}[theorem]{Proposition}
\newtheorem{cor}[theorem]{Corollary}

\newtheorem{intro-thm}{Theorem}

\theoremstyle{definition}

\newtheorem{remark}[theorem]{Remark}

\newcounter{example}

\newcounter{open}

\newcounter{figno}

\newcounter{tableno}

\numberwithin{equation}{section}
\numberwithin{figure}{section}
\numberwithin{table}{section}
\numberwithin{example}{section}

\newcommand{\prob}[1]{\mathbb{P}\left(#1\right)}
\newcommand{\probsub}[2]{\mathbb{P}_{#2}\left(#1\right)}

\newcommand{\solve}{\text{\tt Solve}}
\newcommand{\grow}{\text{\tt Grow}}
\newcommand{\Epeople}{E_{\text{ppl}}}
\newcommand{\Epuzzle}{E_{\text{puz}}}
\newcommand{\Gppl}{G_{\text{ppl}}}
\newcommand{\Gpuz}{G_{\text{puz}}}
\newcommand{\srt}{\text{\tt short}}
\newcommand{\lng}{\text{\tt long}}
\newcommand{\Bin}{\text{\rm Binomial}}
\newcommand{\Poisson}{\text{\rm Poisson}}
\newcommand{\pcsite}{p_c^{\text{site}}}
\newcommand{\cpeople}{\text{\tt coll}}    
\newcommand{\cpuzzle}{\text{\tt link}}   
\newcommand{\final}{\text{\tt final}}

\newcommand{\abs}[1]{\left|#1\right|}

\newcommand{\ER}{Erd\H{o}s-R\'{e}nyi }

\begin{document}
\pagestyle{empty}

\null

\begin{center}\huge
{\bf  Nucleation scaling in jigsaw percolation}
\end{center}

\begin{center}

{\sc Janko Gravner}\\
{\rm Mathematics Department}\\
{\rm University of California}\\
{\rm Davis, CA 95616, USA}\\
{\rm \tt gravner{@}math.ucdavis.edu}
\end{center}
\begin{center}

{\sc David Sivakoff}\\
{\rm Department of Statistics}\\
{\rm Department of Mathematics}\\
{\rm Ohio State University}\\
{\rm Columbus, OH 43210, USA}\\
{\rm \tt dsivakoff{@}stat.osu.edu}
\end{center}

\begin{center}
{\today}
\end{center}

\vspace{0.3cm}

\noindent{\bf Abstract.} Jigsaw percolation is a nonlocal process that iteratively merges
connected clusters in a deterministic ``puzzle graph'' by using connectivity 
properties of a random ``people graph'' on the same set of vertices. 
We presume the \ER people graph with edge probability $p$
and investigate the probability that the puzzle is solved, that is, that the 
process eventually produces a single cluster. 
In some generality, for puzzle graphs with $N$ vertices of degrees about $D$
(in the appropriate sense), 
this probability is close to 1 or small depending on whether $pD\log N$ 
is large or small. The one dimensional ring and two dimensional torus puzzles
are studied in more detail and in many cases the exact scaling of the 
critical probability is obtained. 
The paper 
strengthens several results of Brummitt, Chatterjee, Dey, and Sivakoff
who introduced this model.

\vskip0.3cm

\noindent 2010 {\it Mathematics Subject Classification\/}: 60K35

\vskip0.3cm

\noindent {\it Key words and phrases\/}:  jigsaw percolation, 
nucleation, random graph

\newpage

\addtocounter{page}{-1}

\pagestyle{headings}
\thispagestyle{empty}

\vfill\eject

\section{Introduction}

The two-dimensional discrete torus is the graph with 
vertex set $V=\bZ_n^2=\{0,1\dots,n-1\}^2$ with 
periodic boundary conditions and edges between nearest neighbors.  We imagine 
$V$ as pieces of a puzzle and denote this graph, an instance of a puzzle graph, by $\Gpuz$.
Suppose we have a partially solved puzzle, that 
is, a collection of $\Gpuz$-connected subsets of $V$ 
(also known as {\it clusters\/}) that partition $V$. 
Then we get closer to the complete solution by merging together one 
or more of these clusters. If we have two clusters whose 
union is a connected set in $\Gpuz$, how does the information that they fit 
together, and hence can be merged, get transmitted? The idea introduced 
in \cite{BCDS} is that the knowledge about each piece is held
by a separate person and that the $N=\abs{V}$ people are connected by 
{\it collaboration\/} edges into the people graph $\Gppl$.   This model was proposed as an idealized mechanism by which
people with incomplete knowledge could collaboratively combine their
partial solutions to solve a puzzle.  As in \cite{BCDS}, we assume 
that people connections are sparse and assigned at random. 

\begin{figure}[ht]
\hskip0.35cm\includegraphics[trim=0cm 0cm 0cm 0cm, clip, width=5cm]{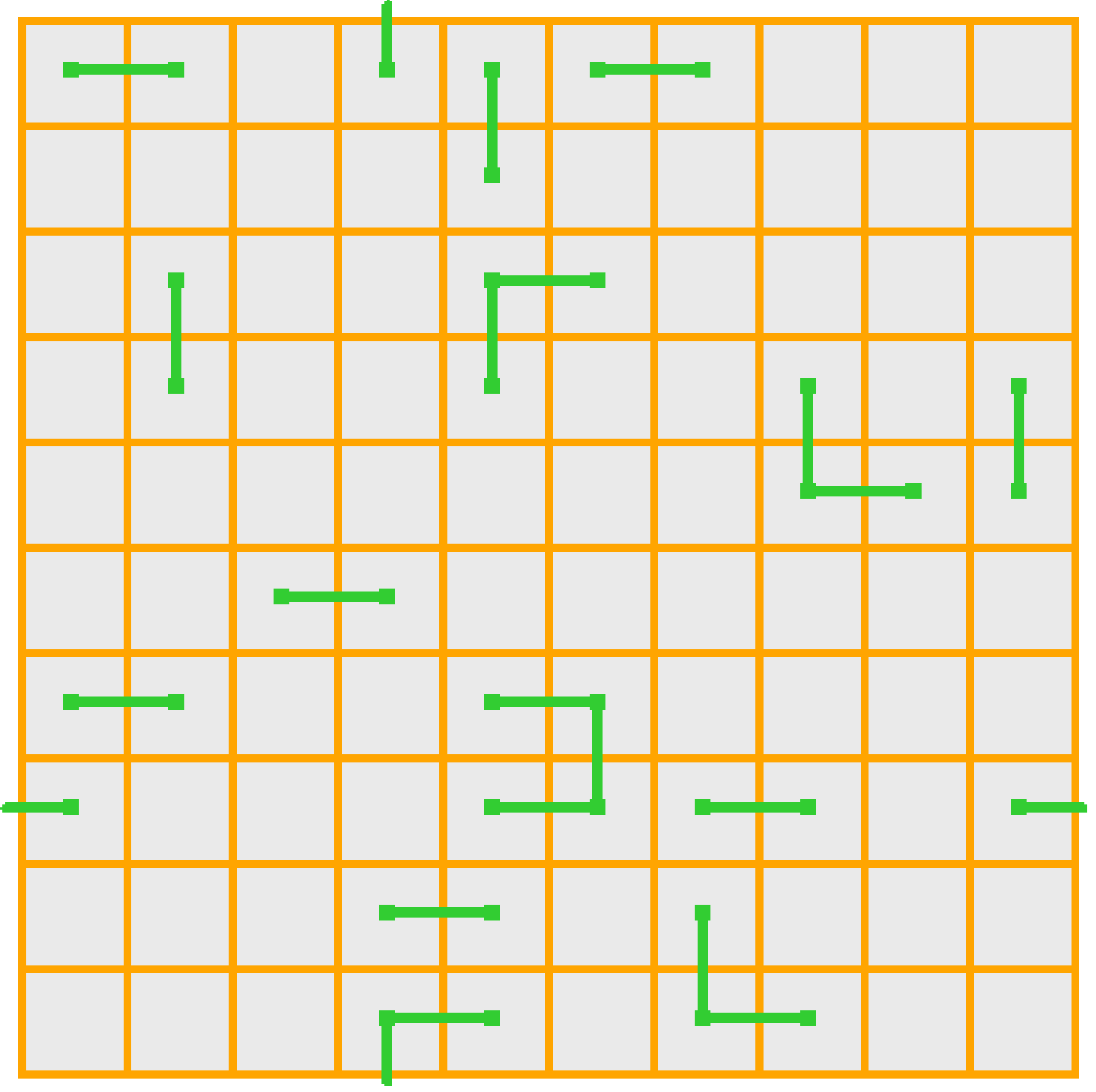}
\includegraphics[trim=0cm 0cm 0cm 0cm, clip, width=5cm]{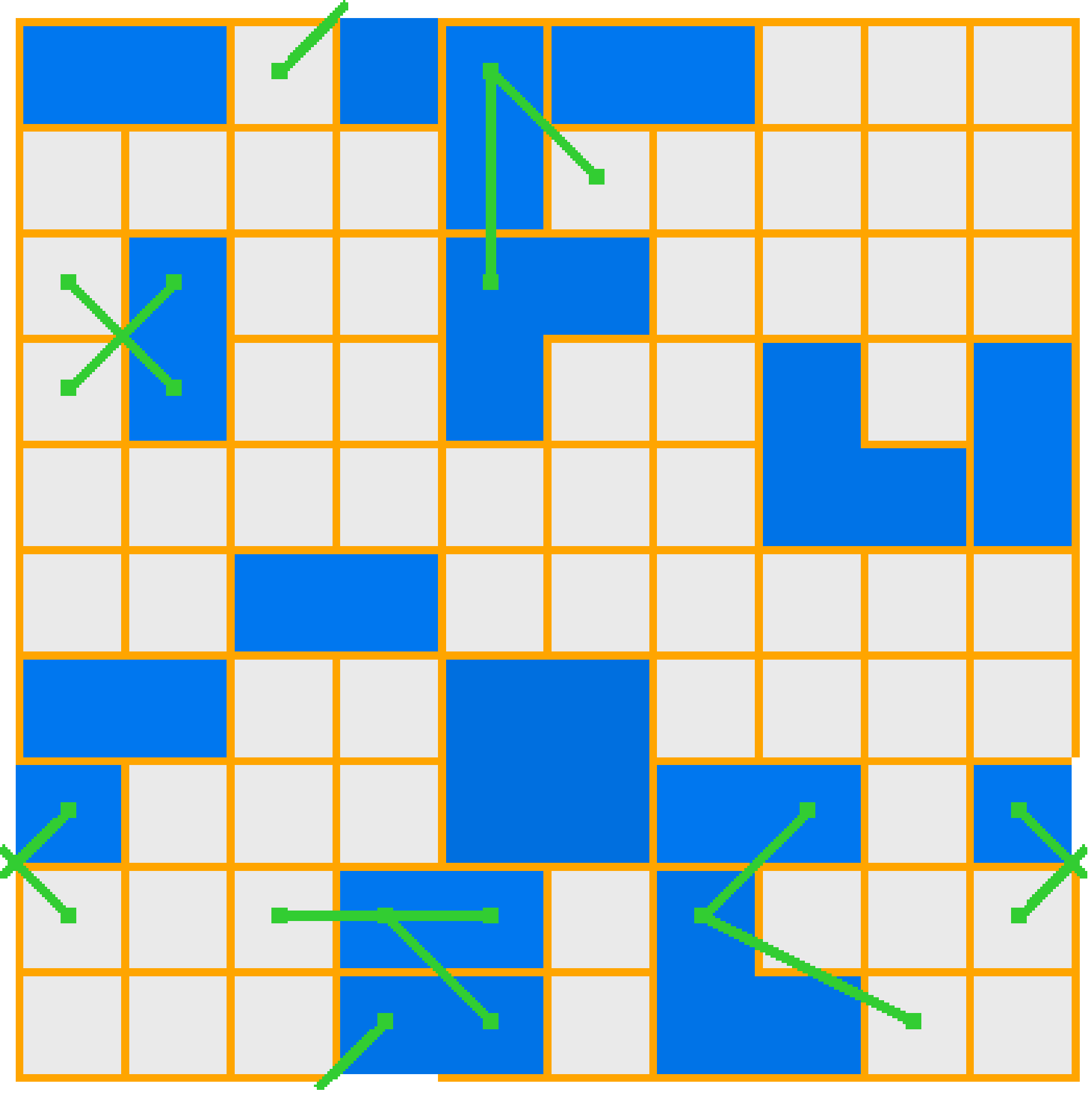}
\includegraphics[trim=0cm 0cm 0cm 0cm, clip, width=5cm]{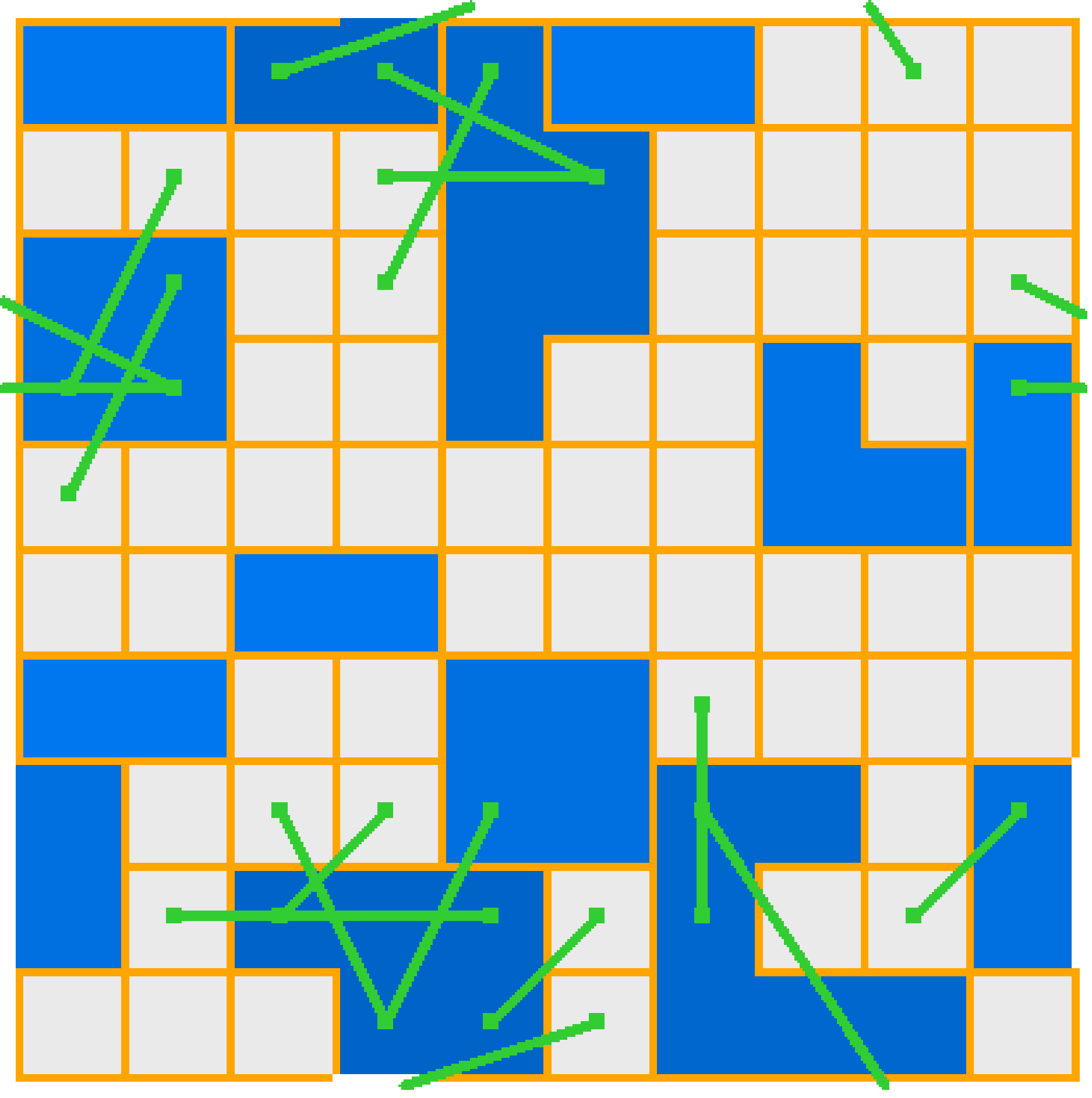}\hfill\newline
\null\hskip0.35cm\includegraphics[trim=0cm 0cm 0cm 0cm, clip, width=5cm]{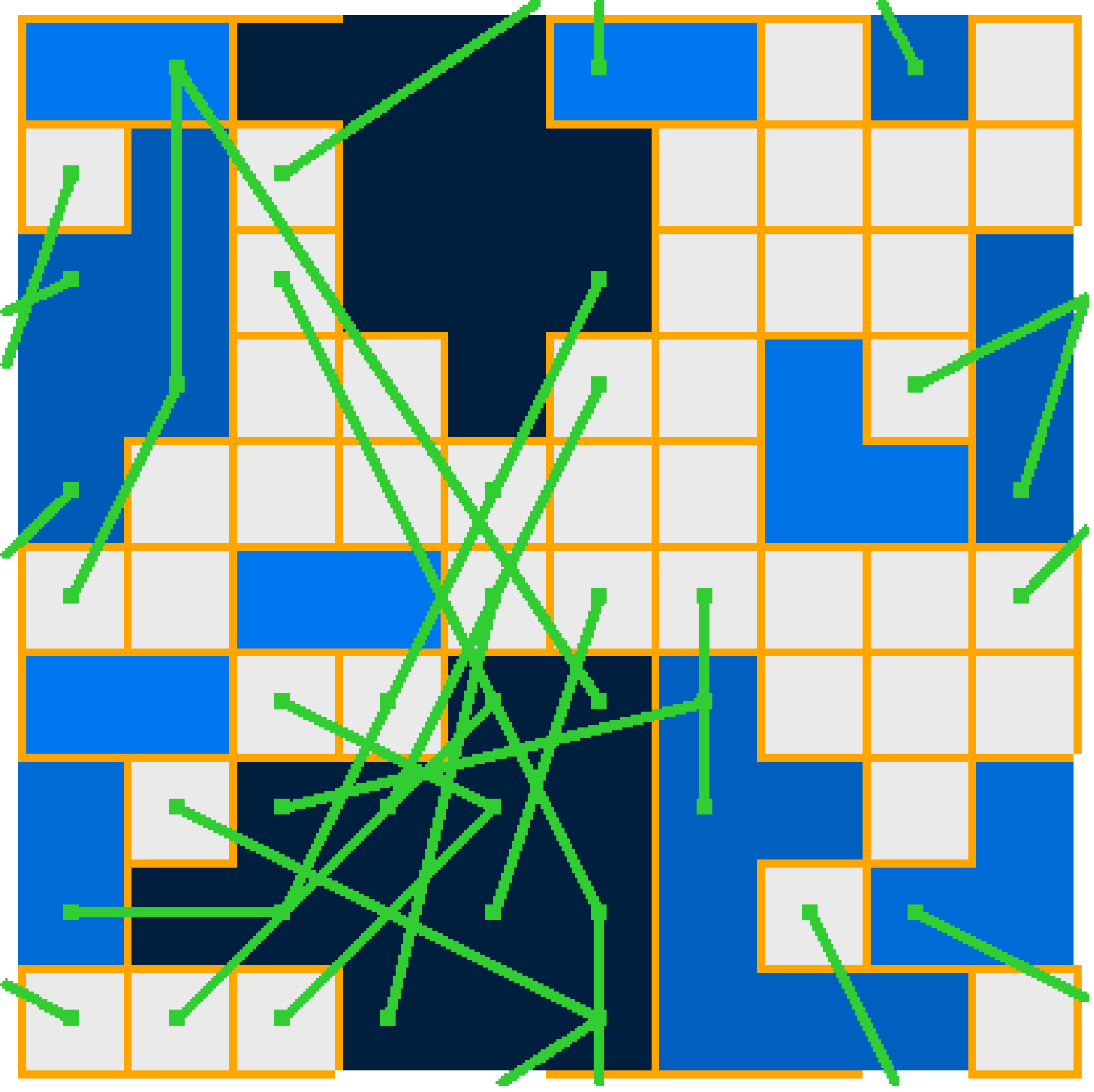}
\includegraphics[trim=0cm 0cm 0cm 0cm, clip, width=5cm]{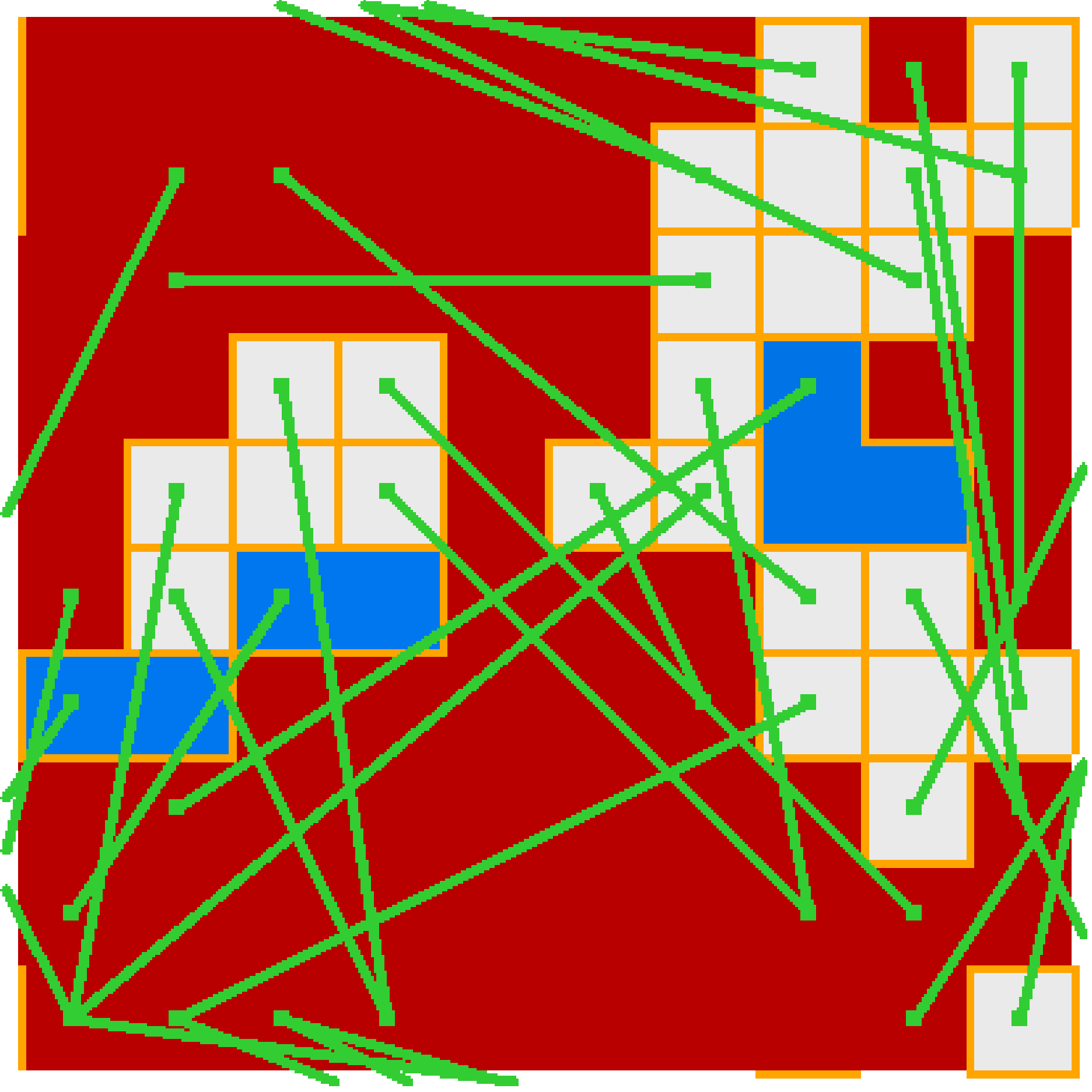}
\includegraphics[trim=0cm 0cm 0cm 0cm, clip, width=5cm]{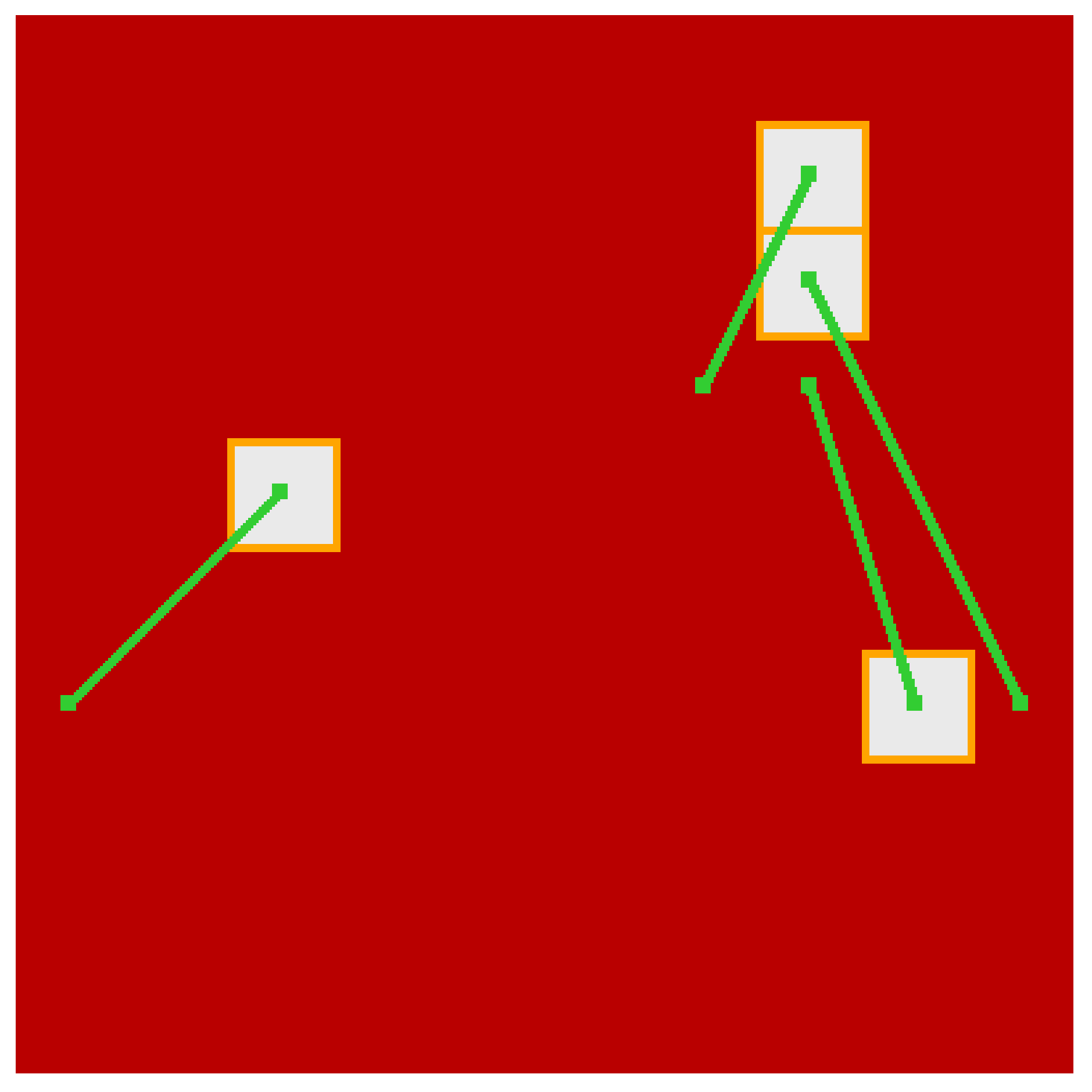}\hfill
\caption{AE jigsaw percolation on $10\times 10$ torus (i.e., square with periodic boundary), 
with $p=0.11$, at times $t=0,\ldots, 5$. 
The clusters are outlined in orange and colored grey-blue-dark blue-red according to their sizes. The edges of $\cG_t$, that decide 
which cluster are merged in the next time step, are depicted in green. The edges  
connect vertices with the $\Gppl$-connection found by the algorithm. Clearly, all vertices 
are in one cluster for the first time at $t=6$.}
\label{AE-figure-small} 
\end{figure}

Our general setting is a sequence of graph pairs $(\Gpuz, \Gppl)$, 
on a common vertex set $V$ whose size $N$  increases with an integer parameter $n$. The dependence on $n$ or $N$ is typically suppressed
in our notation; $N$ will always mean the number of vertices in the graph, and for particular
examples we choose the common parametrization (e.g., the  
two dimensional torus graph $\bZ_n^2$ has $N=n^2$ vertices), while we formulate 
our statements about general graphs in terms of dependence on $N$ rather than $n$.  
The \df{puzzle graph} $\Gpuz=(V, \Epuzzle)$
is (for every $N$) a connected deterministic graph,  while 
we assume throughout that the {\it random\/} \df{people graph}  $\Gppl=(V,\Epeople)$
is an \ER graph on $V$ with  a small edge probability $p$ that also depends on $N$. 

The models we consider retain the general flavor of \cite{BCDS}, with a new ingredient:
how easy it is to discover that a puzzle piece fits to a cluster depends on the number of 
connections of each type between the piece and the cluster. 
A simple implementation of this principle 
leads to a three-parameter model that 
we now introduce.  
 
We say that vertices $v_1,v_2\in V$ are \df{doubly connected} if they are connected in 
both graphs: $\{v_1,v_2\}\in \Epeople\cap\Epuzzle$. For a fixed $v\in V$ and a set $S\subset V$, 
we let $\cpeople(v,S)$ (resp.~$\cpuzzle(v,S)$) be the number of $\Gppl$-neighbors
(resp.~$\Gpuz$-neighbors) of $v$ in $S$, not including $v$.  

We define \df{jigsaw percolation} as a discrete-time dynamics with three threshold parameters: 
\df{verification threshold}  $\sigma\ge 1$, 
\df{link threshold} $\tau\ge 1$, and \df{exemption threshold} $\theta\ge \tau$. 
At each time $t=0,1,2,\dots$, the state of the dynamics is a partition 
$\cP^t=\{W_i^t:i=1,\dots, I_t\}$ of the vertex set, with $\cP^0$ a given partition.  
Given $\cP^t$, $\cP^{t+1}$ is obtained as follows. Construct the graph $\cG_t$ with vertex set $\cP^t$ and unoriented edges
between any $W_i^t$ and $W_j^t$ such that at least one of (J1)--(J3) is satisfied:
\begin{itemize}
\item[(J1)] there are doubly connected vertices $v_1\in W_i^t$ and $v_2\in W_j^t$; 
\item[(J2)] there is a vertex $v_1\in W_i^t$ with $\cpuzzle(v_1, W_j^t)\ge\theta$;
\item[(J3)] there is a vertex $v_1\in W_i^t$ with $\cpeople(v_1, W_j^t)\ge\sigma$ 
and $\cpuzzle(v_1, W_j^t)\ge\tau$.
\end{itemize}
Then, 
\begin{itemize}
\item[(J4)] to obtain $\cP^{t+1}$, merge all sets in $\cP^{t}$ that belong to the 
same connected component of $\cG_t$.  
\end{itemize}

The parameter $\theta$ is akin to the threshold in bootstrap percolation~\cite{AL}, 
in that a vertex will merge with a larger cluster as soon as it has $\theta$ 
$\Gpuz$-neighbors in that cluster.  In the example of $\bZ_n^2$, when $\theta=2$, 
this amounts to filling in puzzle pieces that ``obviously'' fit with the partially 
solved puzzle because they fill in a missing corner.  
Of course, due to the nonlocal nature, other sites may be added to the cluster 
along with the missing corner (namely, those sites that have previously merged with the corner). 
(Another contrast with bootstrap percolation is that we get an essentially equivalent model if we require that the two neighbors, which a vertex needs in a neighboring cluster to join, are diagonally adjacent.) The parameters $\tau$ and $\sigma$ control the levels of redundancy required in the puzzle and people graphs, respectively, for two clusters to merge. 
We say that the event $\solve$ happens if, when $\cP^0$ consists of all singletons, 
the partition eventually gathers all vertices into one set, that is 
$\solve=\{\cP_t=\{V\} \text{ for some t}\}$.

Observe that, for every $t$, 
sets in $\cP^t$ are $\Gpuz$-connected; provided that $\theta=\infty$, they 
are also  $\Gppl$-connected. The model with parameters $\sigma=\tau=1$ and $\theta=\infty$ 
(or equivalently, $\theta$ 
exceeds the maximum degree of $\Gpuz$)
was introduced in \cite{BCDS} as the \df{Adjacent-Edge (AE) jigsaw percolation} 
and we will keep this name. The paper 
\cite{BCDS} mostly analyzes the \df{basic jigsaw percolation} 
in which there is an edge between $W_i^t$ and $W_j^t$
in $\cG^t$ when 
\begin{itemize}
\item[(J5)]there are vertices $v_1,v_1'\in W_i^t$ and $v_2,v_2'\in W_j^t$,
such that $\{v_1,v_2\}\in \Epeople$ and $\{v_1',v_2'\}\in \Epuzzle$. 
\end{itemize}
All our results that apply to AE dynamics 
(Theorems~\ref{intro-lb-general},~\ref{intro-ub-list},~\ref{intro-2d-bounds}, and results in Sections 3 and 4, 
as well as $\sigma=1$ instances of Theorems~\ref{intro-ring-thm} and~\ref{intro-final-time-thm})  
hold for the basic version with unchanged proofs. In fact, as noted in \cite{BCDS}, 
it is an interesting open problem to devise a class of puzzle graphs with significant difference 
in behaviors between the AE and basic versions.  

A small example of solving the torus puzzle 
in the AE case is depicted in Figure~\ref{AE-figure-small} and a larger 
one in Figure~\ref{big-figure}; see Section 10
for a description of algorithms we employ.
The general message of simulations is that $p$ should be large enough so 
that nucleation centers (as in Figure~\ref{big-figure}) appear. In this 
sense, jigsaw percolation is similar to {\it bootstrap percolation\/} \cite{AL, Hol}
and {\it Greenberg-Hastings model\/} \cite{FGG}, in spite 
of the fact that it is non-local. Indeed, to our knowledge the present paper is the first to establish 
scaling of critical probabilities 
and sharp phase transitions using nucleation techniques in a non-local setting.  
 
Typical for nucleation-and-growth models is an {\it order parameter\/}: a function 
of $p$ and $N$ that determines (for large $N$) whether $\probsub{\solve}{p}$ is large or small; 
often there is sharp transition from probability close to $0$ to close to $1$. For example, 
we will prove that for the $\bZ^2_n$ case of Figure~\ref{AE-figure-small}, 
the order parameter is $p\log n$, but a sharp transition in this case remains an open problem.
To make this concept precise, we define 
the \df{critical probability} $p_c=p_c(N)$  by
$$
\probsub{\solve}{p_c}=\frac 12.
$$
We say that there is \df{sharp transition} if 
$\probsub{\solve}{(1-\epsilon)p_c}\to 0$ and $\probsub{\solve}{(1+\epsilon)p_c}\to 1$
as $N\to\infty$, for any $\epsilon>0$.
It is expected that under general conditions there is sharp transition \cite{FK}.
We will prove this for some examples 
in which the asymptotic behavior of $p_c$ can be determined exactly. The general 
results in \cite{FK} (and subsequent work) cannot be used as they depend on 
symmetry of random bits. This in our case clearly fails as, for example, 
$\Gppl$-edges between $\Gpuz$-neighbors do not play exactly the same role 
as  other $\Gppl$-edges. 
We refer to \cite{BCDS} for much more background and intuition. 
We now state our main results, which are divided into three categories
in subsections below. 


\subsection{Results for general puzzle graphs}

Notably, the asymptotic order of $p_c$ can be 
determined in some generality.  
In this subsection, we assume 
the puzzle graph $\Gpuz$ has maximum degree $D$, which may depend on $N$.
The proof of the following theorem is given in Section 3.

\begin{intro-thm}\label{intro-lb-general}
Assume AE dynamics and that
$p =\mu/(D\log N)$ for a constant $\mu\le 1/30$. Then $\probsub{\solve}{p} \to 0$.
\end{intro-thm}
 
Theorem 2 from \cite{BCDS} demonstrates that for the AE dynamics
$\probsub{\solve}{p} \to 1$ if 
$p\ge C/\log N$, for an absolute constant $C$.  The next theorem provides
a more precise result for some well-known vertex-transitive graphs: 
together with Theorem~\ref{intro-lb-general} it implies that 
$p_c$ scales as $1/(D\log N)$ in these cases. On the other hand, 
in Section 4 we will exhibit a vertex-transitive example for which
this scaling does not hold. Section 4 also contains a general method
used to prove results such as Theorem~\ref{intro-ub-list}.
 
\begin{intro-thm}\label{intro-ub-list} 
Assume AE dynamics and $p =\mu/(D\log N)$ for a constant $\mu$. For each of the 
following vertex-transitive graphs, there exists a universal constant $C$ such that $\mu\ge C$ implies 
$\probsub{\solve}{p} \to 1$: $d$-dimensional torus $\bZ_n^d$ with lattice edges;  
range-$r$ 
two-dimensional graph on $\bZ_n^2$ with neighborhood of $x$ given by $\{y: ||x-y||_\infty\le r\}$; 
hypercube with vertex set $\{0,1\}^n$;
and $d$-dimensional Hamming graph with vertex set $\bZ_n^d$. 
\end{intro-thm}

An important question we attempt to answer in various contexts 
is the following: how costly is it to require a large number of 
verifications in the people graph? 
Our next result, proved in Section 6, clarifies the general answer 
for the most natural setting whereby we keep the AE parameters $\tau=1$, $\theta=\infty$,  
but assume $\sigma$ is large. It turns out that the number 
of people connections required to solve the puzzle then increases as the 
square of $\sigma$. 

\begin{intro-thm}\label{intro-large-sigma-thm} Assume that $\tau=1$, $\theta=\infty$, 
and $\sigma$ is arbitrary. Assume also that the degree $D$ is bounded by a constant
independent of $N$. Then, 
for $N\ge N_0(\sigma)$, $p_c$ is between two constants times $\sigma^2/\log N$. 
\end{intro-thm}

\subsection{Results for the ring graph}

We next turn to more precise results for low-dimensional puzzle lattices.
As pointed out in \cite{BCDS}, the one-dimensional {\it ring\/} puzzle 
with $V = \bZ_n$ is already of interest. 
By exploiting remarkable similarity to two-dimensional 
bootstrap percolation (see Section 5 for details), we prove Conjecture 2 
of \cite{BCDS} and shed light on Open Problem 1 in the same paper. 
For $\sigma\ge 1$, we let
$$
g_\sigma(x)=-\log \prob{{\mathrm{Poisson}}(x)\ge \sigma} 
$$
and 
$$
\lambda_\sigma=\int_0^\infty g_\sigma(x)\,dx.
$$

\begin{intro-thm}\label{intro-ring-thm} Assume the ring puzzle graph 
with $\tau=1$, $\theta=\infty$, and arbitrary $\sigma\ge 1$.
As $n\to\infty$, 
$$
p_c\log n\to\lambda_\sigma,
$$
with sharp transition.
\end{intro-thm}

\begin{intro-thm}\label{intro-final-time-thm} In the context of Theorem~\ref{intro-ring-thm}, 
assume $p\sim\lambda/\log n$, for 
some $\lambda>0$. Let $T_f=\min\{t:\cP_{t+1}=\cP_t\}$ be the time when the jigsaw dynamics stops. 
Then, in probability, 
$$
\begin{cases}
\limsup_{n\to\infty} \frac{T_f}{\log n}<\infty &\text{if $\lambda<\lambda_\sigma$}\\
\lim_{n\to\infty}\frac{\log T_f}{\log n}=\frac{\lambda_\sigma}{\lambda} &\text{if $\lambda>\lambda_\sigma$}
\end{cases}
$$
\end{intro-thm}
Roughly then, for large $n$, $(\log n)^{-1}\log T_f$ as a function of $p\log n$ 
vanishes on $[0,\lambda_\sigma)$, has a discontinuous jump to $1$ at $\lambda_\sigma$ 
and then decreases to 0 as the inverse first power. For a comparison with 
simulations, see Figure~5b in \cite{BCDS}.

\subsection{Results for the two dimensional torus}

The bulk of the paper is devoted to the case of two-dimensional lattice torus $\Gpuz$
with $V=\bZ_n^2$, which will be assumed in Theorems~\ref{intro-2d-bounds},~\ref{intro-theta2-thm}, 
and \ref{intro-tau2-theorem}. We begin with the AE dynamics, for which
Theorems~\ref{intro-lb-general} and~\ref{intro-ub-list} imply that the 
order parameter is $p\log n$, as previously announced. 
While we are unable to prove sharp transition, we will at least give upper and 
lower bounds within a factor of $10$. We suspect the lower bound is the 
cruder of the two; see Section 7 for a proof. 

\begin{intro-thm}
\label{intro-2d-bounds} Assume the AE dynamics.
For a large enough $n$, 
$$\frac{0.0388}{\log n}< p_c < \frac {0.303}{\log n}.
$$
\end{intro-thm} 

The two-dimensional torus is the simplest instance for which we can investigate 
the dependence on two puzzle graph thresholds $\theta$ and $\tau$.
It it is not hard to see that for this $\Gpuz$ there are essentially only three interesting cases: 
$\tau=1, \theta=\infty$; $\tau=1$, $\theta=2$; and $\theta\ge \tau=2$. The 
first case is covered by Theorems~\ref{intro-2d-bounds} and~\ref{intro-large-sigma-thm}, 
while the other two are addressed in our next two results. 
Many open problems 
remain for other puzzle graphs; see the Open problems section at the end of the paper.

Our most substantial result is about the parameter choice $\tau=1$ and $\theta=2$.
In this case, corners are fit automatically, but non-corner 
pieces require $\sigma\ge 1$ verifications. By contrast to Theorem~\ref{intro-large-sigma-thm}, 
and perhaps surprisingly, $\sigma$ now affects the {\it power\/} of $\log n$ 
in the critical scaling. Change of the order parameter without a change 
in the underlying geometry appears to be a novel phenomenon. 
We let $g(x)=g_1(x)=-\log(1-e^{-x})$ and 
$$
\nu_\sigma=\int_0^\infty g\left(\frac{x^{2\sigma+1}}{\sigma!}\right)\, dx
=\frac{(\sigma!)^{1/(2\sigma+1)}\Gamma(\frac 1{2\sigma+1})\zeta(\frac {2\sigma+2}{2\sigma+1})}
{2\sigma+1}.
$$
For example, when $\sigma=1$, 
$\nu_1=\frac{\Gamma(\frac 13)\zeta(\frac 43)}{3}\approx 3.216$ and the next theorem 
implies that 
transition occurs at $p(\log n)^3=\nu_1^3\approx 33.25$. See Figure~\ref{big-figure}
for an illustration and Section 8 for a proof.
 
\begin{intro-thm}\label{intro-theta2-thm}  Assume $\tau=1$, $\theta=2$ and $\sigma\ge 1$. 
As $n\to \infty$, 
$$
p_c(\log n)^{2+\frac 1\sigma}\to \nu_\sigma^{2+\frac 1\sigma},
$$
with sharp transition. 
\end{intro-thm}

The final interesting case has $\tau=2$, and arbitrary
$\theta$ and $\sigma$. 
The asymptotic scaling of the critical probability is always 
$1/\log n$, but the only instance we are able to identify the constant factor
is when $\theta=2$ and the dynamics does not depend on $\sigma$. We give 
a proof in Section 9 and, 
again, Figure~\ref{big-figure}
provides an illustration.

\begin{intro-thm}\label{intro-tau2-theorem} Assume $\theta\ge \tau=2$, 
and $\sigma\ge 1$. Then,
\begin{equation}\label{intro-tau2-general}
\frac{\pi^2}{6}\le \liminf_{n\to\infty} p_c\log n\le 
\limsup_{n\to\infty} p_c\log n\le \frac {\pi^2}{6}+\frac 
12\int_0^\infty g_\sigma(x)\, dx.
\end{equation}
If $\tau=\theta=2$, then
\begin{equation}\label{intro-tau2-theta2}
p_c\sim\frac {\pi^2}{6}\cdot \frac{1}{\log n}
\end{equation}
as $n\to\infty$, with sharp transition.
\end{intro-thm}

The lower bound in (\ref{intro-tau2-general}) can be improved for large $\sigma$, 
as Theorem~\ref{intro-large-sigma-thm} implies that it can be replaced by a bound on the order 
$\sigma^2$. We do not know whether the upper bound in (\ref{intro-tau2-general}) (which is also on the
order $\sigma^2$) can be improved.

\begin{figure}
\centering\includegraphics[trim=0cm 0cm 0cm 0cm, clip, width=7.8cm]{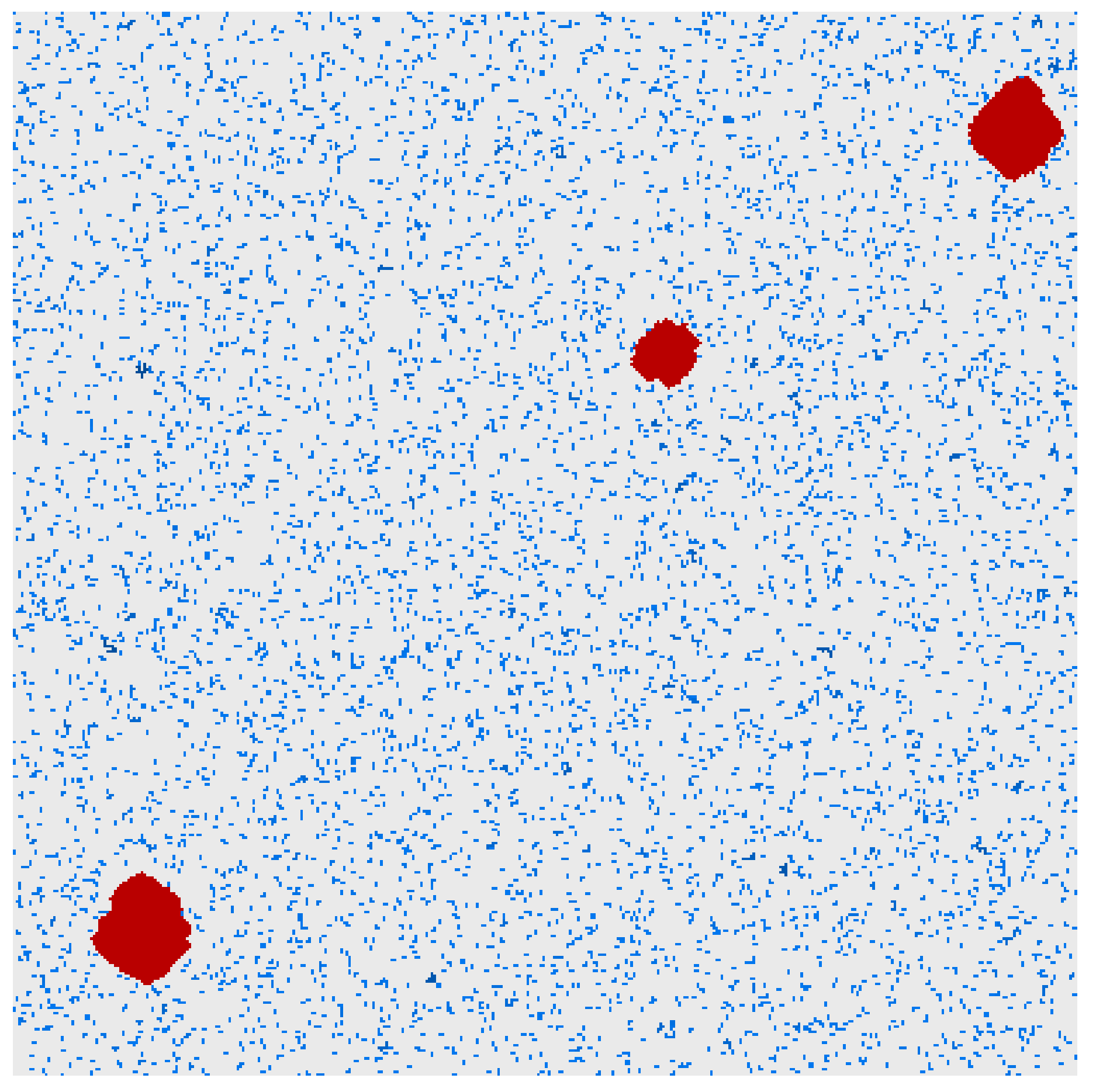}
\includegraphics[trim=0cm 0cm 0cm 0cm, clip, width=7.8cm]{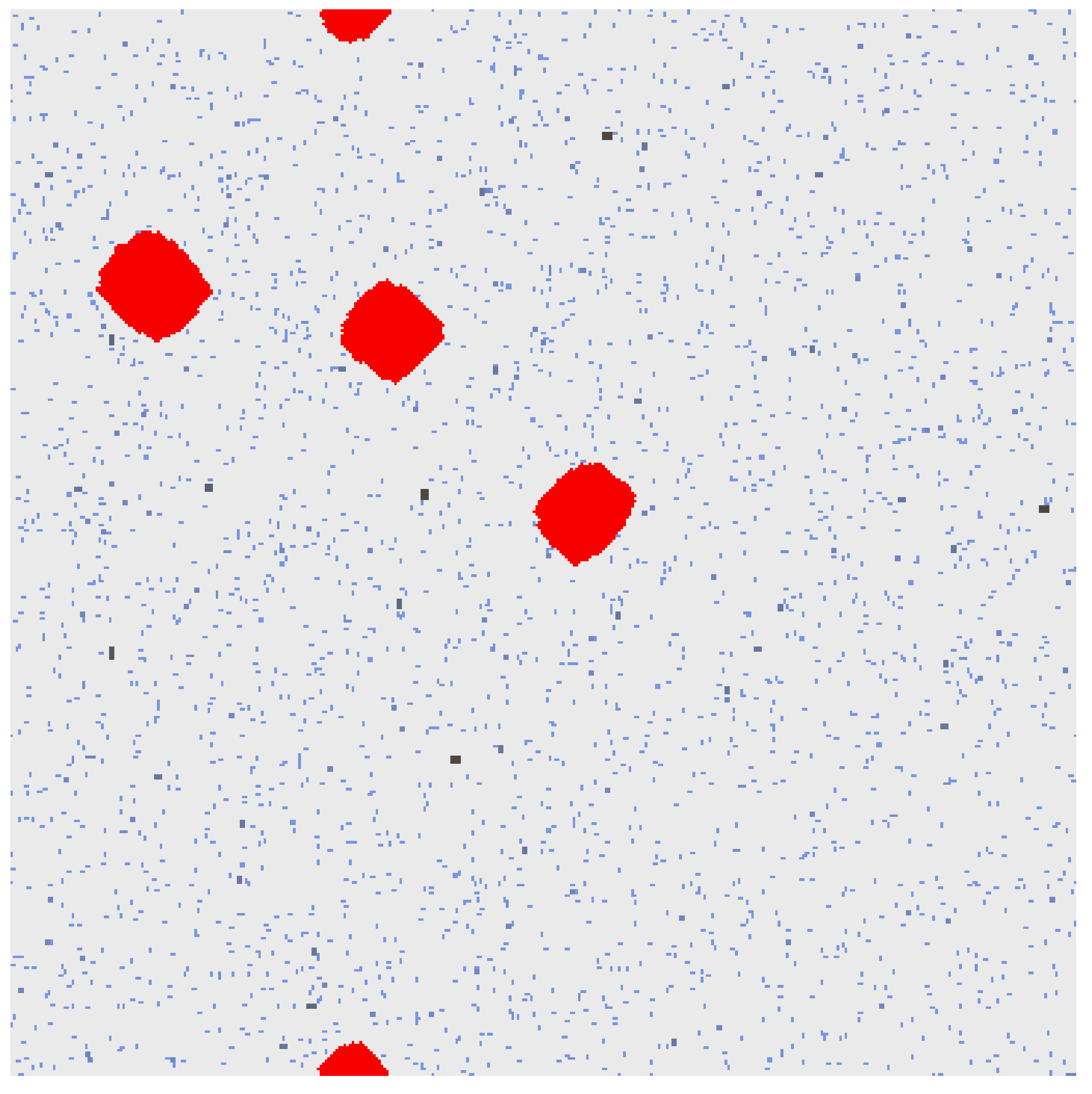}
\includegraphics[trim=0cm 0cm 0cm 0cm, clip, width=7.8cm]{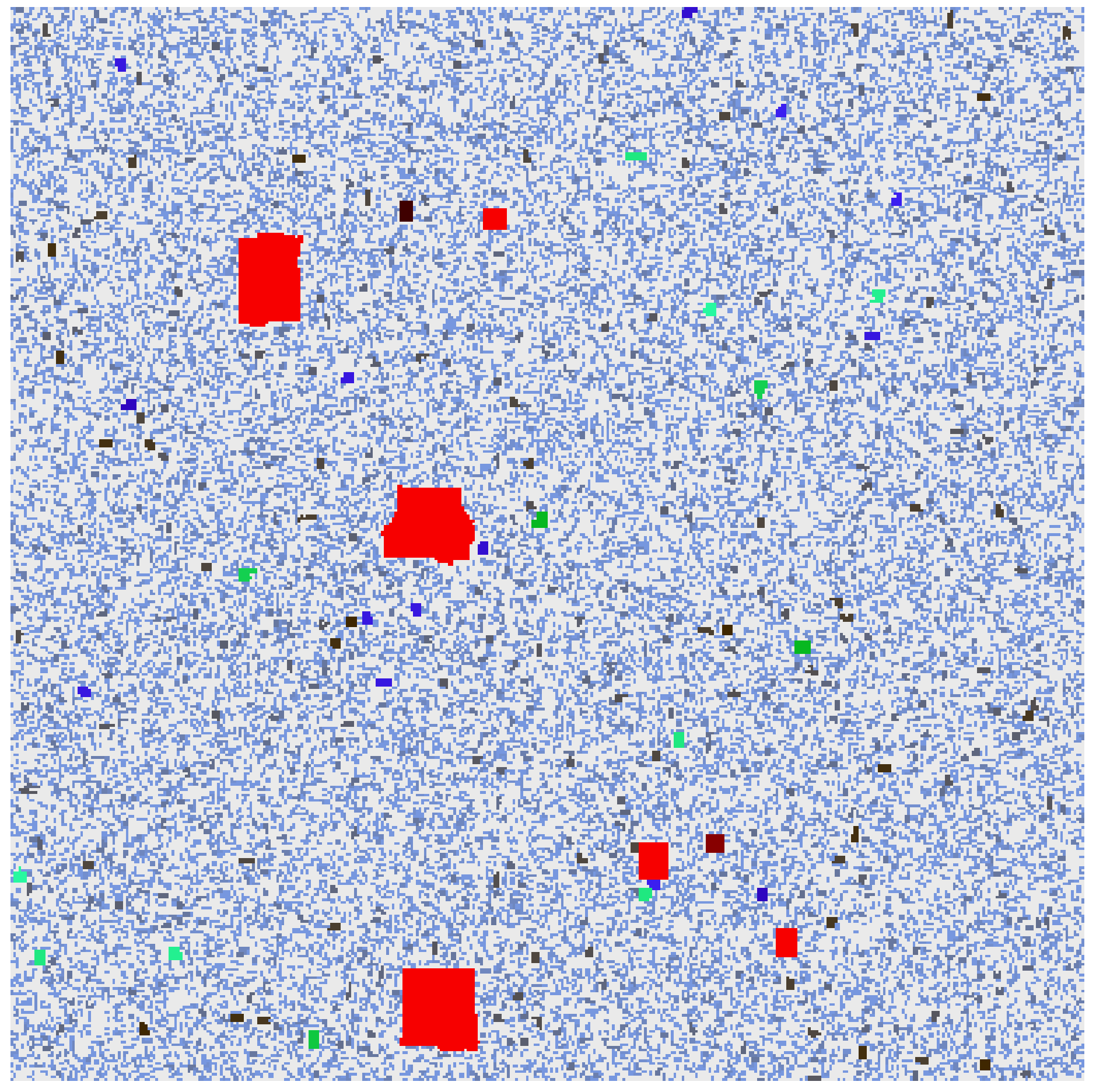}
\caption{Jigsaw percolation on $400\times 400$ torus. Top left: AE version ($\sigma=\tau=1$, $\theta=\infty$) 
with $p=0.021$, at time $t=31$. Top right: $\sigma=\tau=1$, $\theta=2$ with $p=0.009$, at $t=31$.
Bottom: $\sigma=1$, $\tau=2$, $\theta=\infty$, with $p=0.11$ at $t=91$. 
These pictures illustrate the nucleation and metastability of the dynamics: 
most of the space is divided into small clusters (color-coded by size with grey singletons). 
A few favorable local configurations generate large (red) clusters that grow unstoppably 
and result in $\solve$.}
\label{big-figure} 
\end{figure}

We remark that for random puzzle graphs, a related result was proven in \cite{Slivken}, where
it is assumed that both the people and puzzle graphs are \ER with probabilities $p$ and $p_{\text{puz}}$, 
respectively, with $p\wedge p_{\text{puz}} \geq (1+\epsilon) \log N / N$ for some $\epsilon>0$. 
Then it is shown that the probability 
of solving the puzzle is close to zero if $p\cdot p_{\text{puz}} \leq c/(N\log N)$ and is close to one if $ p\cdot p_{\text{puz}} \geq C\log\log N /(N \log N)$, for some constants $c,C>0$. 

The rest of the paper is organized as follows.
We begin with Section 2 that contains more formal definitions and 
some useful observations. Sections 3--9 are devoted to proofs to 
the above theorems. Section 10 contains a discussion of computational aspects of simulation 
algorithms, and the paper is concluded by a list of intriguing open problems.

\section{Preliminaries}

We say that a given partition $\cP$ is \df{inert} if jigsaw percolation 
started at $\cP^0=\cP$ results in $\cG^0$ with no edges and thus $\cP^1=\cP$. Clearly, 
for any $\cP^0$, $\cP^t$ is inert for some $t$; we call this the 
\df{final} partition started from $\cP_0$.  Note that the final partition also 
depends on $\Gppl$.

When not explicitly stated otherwise, 
the initial partition $\cP^0$ will consist of all singletons and in this case we 
denote the final partition as $\final$.  
 
For a given set $S\subset V$, denote its \df{outside boundary} by 
$\partial_o(S)=\{v\notin S: \{v,v'\}\in \Epuzzle\text{ for some }v'\in S\}$. 

\begin{prop} \label{inert-int} Assume that $\{\cP_j\}$ is a finite collection of inert partitions
of $V$, and let the partition $\cP$ consist of all non-empty intersections $\cap_j W_{j}$ for arbitrary $W_{j}\in \cP_j$.
Then $\cP$ is also inert. 
\end{prop}
\begin{proof}
A partition $\cP$ is inert if and only if for every $W\in \cP$ and every vertex
$v\in \partial_o(W)$ all of the following hold:
$v$ is not doubly connected to any vertex in $W$; $\cpuzzle(v, W)<\theta$;
$\cpeople(v, W)<\sigma$ or $\cpuzzle(v, W)<\tau$. To show this holds for 
the so defined $\cP$, pick $v\in \partial_o(\cap_jW_{j})$ for arbitrary $W_{j}\in \cP_j$ such that their intersection is nonempty.  Then 
$v\in \partial_o (W_{j_0})$ for some $j_0$ and 
$\cpuzzle(v, \cap_jW_{j})\le \cpuzzle(v, W_{j_0})<\theta$.
Other verifications are similar.
\end{proof}

By Proposition~\ref{inert-int}, for any partition $\cP^0$, there exists an inert partition $\langle\cP^0\rangle$,
which is the finest of all inert partitions $\cP$ such that 
$\cP^0$ is finer than $\cP$. 

We call a dynamics on partitions a \df{slowed-down} jigsaw percolation if (J4) is 
replaced by the following. 
\begin{itemize}
\item[(J6)] If $\cG_t$ has no edges, $\cP^{t+1}=\cP^t$; otherwise use some rule 
to choose any {\it nonempty\/} 
subset of edges of $\cG_t$ to form a graph $\cG_t'$, then merge all sets in $\cP^{t}$ which are in the 
same connected component of $\cG_t'$ to obtain $\cP^{t+1}$.  
\end{itemize}

The following corollary, which is now immediate, in particular states 
that the final partition is independent
of the slowed-down version. 

\begin{cor}\label{inert-final}
For any slowed-down jigsaw percolation, and any $\cP^0$, there exists a 
$t$ for which $\cP^t=\langle\cP^0\rangle$.
\end{cor}

We recall that, 
for a given graph $\Gpuz$ and a choice of parameters $\sigma$, $\tau$, and $\theta$, we 
let $\solve$ be the event $\{\final=\{V\}\}$ that the jigsaw percolation eventually 
gathers all vertices in a single cluster. (Recall also that the random partition $\final$ assumes the 
default partition $\cP^0$ that consists of singletons.) Most of this paper will 
be concerned with estimating $\probsub{\solve}{p}$ for particular choices of $p$. 
Figure~\ref{big-figure} and analogy 
with other nucleation processes \cite{AL, Hol, GH, GHM, FGG}, suggest that the 
dominant mechanism in jigsaw percolation is growth from
a single center into undisturbed environment. This approach yields
a good, and often optimal, lower bound on the probability 
of $\solve$, as we will see. 
Thus we introduce the following 
\df{local jigsaw percolation}.
As before, we assume that $\Gpuz$ is a connected deterministic graph on $V$ and 
$\Gppl$ is a random graph in which each pair of vertices is independently connected with 
a probability $p$, but here $V$ is typically infinite. Fix a \df{center} $v_0\in V$. 
The dynamics iteratively determines random sets $V_0\subset V_1\subset\ldots\subset V$.
Let $V_0=\{v_0\}$. For $t\ge 0$, $V_{t+1}\supset V_t$ is obtained by 
adjoining to $V_t$ any $z\in V$ which is either: doubly connected 
to a point in $V_t$; or $\cpuzzle(z,V_t)\ge \theta$; or 
($\cpuzzle(z,V_t)\ge \tau$ and $\cpeople(z,V_t)\ge \sigma$). Define the event
$$
\grow=\{\cup_t V_t=V\}. 
$$
We will use comparison with the local version when $\Gpuz$ is the two-dimensional torus 
$\bZ_n^2$. In this case, the corresponding local process 
is on the first quadrant $\bZ_+^2$ with center $(0,0)$. As we will see, 
$\probsub{\solve}p\approx n^2\probsub{\grow}p$ in the relevant regime. 

For a graph $G = (V,E)$ and a subset of vertices, $A\subset V$, 
let $G^A$ denote the subgraph of $G$ induced by $A$.  
That is, $G^A$ is the graph with vertex set $A$ and edge set 
$E^A = \{\{u,v\}\in E : u,v \in A\}$.

We say that a subset of vertices, $A\subset V$, is \df{internally solved} \cite{BCDS} 
if the jigsaw percolation process with people graph $\Gppl^A$ 
solves the puzzle graph $\Gpuz^A$.  We denote this event as $\solve_A$. 
Similarly, for a partition $\cP^0$ of $A$, we denote by $\langle\cP^0\rangle_A$ 
the final partition obtained by running the jigsaw percolation 
with the two induced graphs,
and let $\final_A=\langle\cP^0\rangle_A$ when $\cP^0$ is the set of singletons of $A$.

For two sets $A\subset A'\subset V$, we let  
$
D(A,A')
$
be the event that $\langle \cP^0\rangle_{A'}=\{A'\}$ when the initial partition is $\cP^0=\{A, \{v\}: v\in A'\setminus A\}$. Therefore, 
$$
\probsub{\solve_{A'}\mid\solve_{A}}{p}=\probsub{D(A,A')}{p},
$$
and we may think of $D(A,A')$ as the event that jigsaw percolation 
internally solves $A'$ provided it has already solved $A$.

We now state a key observation; see \cite{AL} for the analogous result 
for bootstrap percolation. 

\begin{lemma}\label{is-double} For any slowed-down jigsaw percolation, 
all sets in the partition at any time are internally solved. 
If $\solve$ happens, then for any $k \leq N/2$ there exists an $A\subset V$, with  
$|A|\in[k, 2k]$, such that $\solve_A$ happens.  
\end{lemma}

\begin{proof} The first claim is a simple observation. For the second claim, 
consider a slowed-down jigsaw percolation where at each step the 
graph $\cG_t'$ in (J6) has at most one edge, so that if the process does not
stop exactly two clusters merge. In this version, the size of the largest 
cluster can at most double in a single step. 
\end{proof}

We call a set $A\subset V$ of vertices \df{unstoppable} if every vertex $v\in V\setminus A$ 
is $\Gppl$-connected to at least $\sigma$ vertices in $A$. The following simple observation is
frequently used. 

\begin{lemma}\label{unstoppable} Assume $\tau=1$. For any $A\subset V$, 
$$
\solve_A\cap \{A\text{ is unstoppable}\}\subset \solve.
$$
\end{lemma}

\begin{lemma}\label{unstoppable-size}
Assume that $S\subset V$ is a set of size at least $\alpha\frac{\log N}{p}$, 
for $\alpha>\sigma$. Then, 
$$
\probsub{\text{$S$ is unstoppable}}{p}\ge 1-3\sigma N^{1-\alpha/\sigma}.
$$
\end{lemma}

\begin{proof} If $|S|=k$, 
\begin{equation}\label{unstoppable-size-eq1}
\begin{aligned}
\probsub{\text{$S$ is not unstoppable}}{p}&\leq
(N-k)(1-\prob{\Bin(k,p)\ge \sigma})\\
&\le N (1-\prob{\Bin(\lfloor k/\sigma\rfloor,p)\ge 1}^\sigma)\\
&= N(1-(1-(1-p)^{\lfloor k/\sigma\rfloor})^\sigma)\\
&\le N(1-(1-e^{-p\lfloor k/\sigma\rfloor})^\sigma)\\
&\le \sigma N e^{-p\lfloor k/\sigma\rfloor}.
\end{aligned}
\end{equation}
\end{proof}

Another useful simple observation concerns ``dividing up'' the edge probability in 
$\Gppl$.  

\begin{lemma}\label{dividing-up} If $p_j\ge 0$, then 
the union of independent $\Gppl$-graphs with edge probabilities $p_j$ is 
stochastically dominated by the $\Gppl$-graph with edge probability $1\wedge\sum_jp_j$.
\end{lemma}

The following elementary lemma is useful when estimating large deviation 
probabilities of a binomial random variable with small expectation. 

\begin{lemma}\label{binomial-easy-lem}
For all $m$, $k$, $\beta$, 
\begin{equation}
\label{binomial-easy}
\prob{\Bin(m,\beta)\ge k)}\le \binom{m}{k}\beta^k\le\left(\frac{3m\beta}{k}\right)^k.
\end{equation}
\end{lemma}

In Section 4, we also need the following large deviation bound. 

\begin{lemma}\label{binomial-ld}
If $p$ is small enough, $P(\Bin(n,p)\le np/2)\le \exp(-np/7)$. 
\end{lemma}

If we have an event $A$ (that depends on $N$), and $\probsub{A}{p}\to 1$ 
as $N\to\infty$, we say that $A$ occurs \df{asymptotically almost 
surely} (\df{a.~a.~s.}). 

Finally, we remark that we often omit integer parts when we specify integer 
quantities such as lengths and rectangle dimensions.

\section{General graphs: lower bound} 

Assume the puzzle graph $\Gpuz$ has maximum degree $D$, which may depend on $\abs{V}=N$.   
We will prove the following result, which implies Theorem~\ref{intro-lb-general}. We assume 
the AE dynamics, that is, parameters $\tau=\sigma=1$, $\theta=\infty$, 
throughout this section.

\begin{theorem}\label{lb-general}
If $p =\mu/(D\log N)$ and $\mu < \min\{2e^{-(3+\eta)}, e^{-(5+\eta)/2}\}$ where $\eta = \limsup \frac{\log D}{\log N}$, then $\probsub{\solve}{p} \to 0$.
\end{theorem}

\begin{remark}
Notice that $\eta \in [0,1]$, and the two expressions in the constraint on $\mu$ are equal when $\eta = 2 \log 2-1$.
\end{remark}

When combined with Theorem 2 of~\cite{BCDS}, Theorem~\ref{lb-general} gives the following corollary. 

\begin{cor}\label{bounded-degree-corollary}
If $\Gpuz$ has maximum degree bounded above by $D$ as $N\to\infty$, then $p_c$ is bounded between two  constants (depending only on $D$) times $1/\log N$.
\end{cor}

The proof of the Theorem~\ref{lb-general} appears after the next two important lemmas.

\begin{lemma}
\label{solveA-lem}
Suppose $A\subset V$ is a set of vertices such that $\abs{A} = \alpha \log N$ and $\Gpuz^A$ is connected. If $p = \mu /(D \log N)$ with $\mu < 2D/\alpha$, then
\begin{equation*}
\probsub{\solve_A}{p} \leq \frac{2D}{\alpha \mu} N^{\alpha(1-\alpha\mu/(2D) - \log(\frac{2D}{\alpha\mu}))}.
\end{equation*}
\end{lemma}

\begin{proof}
Observe that in order to solve any connected puzzle, the people graph must at least be connected, and any connected graph on $\abs{A}$ vertices must have at least $\abs{A}-1$ edges, so
\begin{equation}
\label{solveA bound}
\probsub{\solve_A}{p} \leq \probsub{\Gppl^A \text{ is connected}}{p} \leq \probsub{\abs{\Epeople^A} \geq \abs{A}-1}{p}.
\end{equation}
The distribution of $\abs{\Epeople^A}$ is stochastically dominated by Binomial($\abs{A}^2/2, p$), so for any $\theta>0$ we have
\begin{align*}
\probsub{\abs{\Epeople^A} \geq \abs{A}-1}{p} & = \probsub{e^{\theta \abs{\Epeople^A}} \geq e^{\theta(\abs{A}-1)}}{p}\\
&\leq e^{-\theta(\abs{A}-1)} \left[1 + \left(e^{\theta}-1\right)p\right]^{\abs{A}^2/2}\\
&\leq \exp\left[(e^\theta-1)\abs{A}^2 p / 2 - \theta \abs{A} + \theta  \right].
\end{align*}
Substituting $\theta = \log(2D/(\alpha \mu)) >0$,  and using inequality (\ref{solveA bound}) gives the result.
\end{proof}

\begin{lemma}
\label{clusters-lem}
Fix a vertex $v\in V$, and let
$$C(v,k) := \{A\subset V : v\in A, \abs{A}=k, \Gpuz^A \text{ is connected}\}.$$
Then,
$$
\abs{C(v,k)} \leq (eD)^k
$$
\end{lemma}

The proof follows an argument of Kesten (\cite{Kes1}, pg.~85).

\begin{proof}
Consider independent site percolation on $\Gpuz$ with vertex probability $1/D$.  The probability that $A$ is a connected component in the site percolation graph is
\begin{align*}
(1/D)^k (1-1/D)^{\abs{\partial_o A}} \geq (1/D)^k (1-1/D)^{(D-1)k},
\end{align*}
where the inequality follows because every vertex in $A$ has at most $D$ neighbors, at least one of which is in $A$.  Summing the probability that $A$ is a connected component in the site percolation graph over all sets $A\in C(v,k)$ gives the probability that $v$ is in a site percolation cluster of size $k$, which of course is at most $1$. Therefore,
$$
\abs{C(v,k)} (1/D)^k (1-1/D)^{(D-1)k} = \sum_{A\in C(v,k)} (1/D)^k (1-1/D)^{(D-1)k} \leq 1.
$$
This yields
\begin{align*}
\abs{C(v,k)} &\leq [D(1-1/D)^{-D+1}]^k  = D^k \left[ 1 + \frac{1}{D-1} \right]^{(D-1)k} \leq D^k e^k,
\end{align*}
where in the last inequality we used $1+x\leq e^x$.
\end{proof}

\begin{proof}[Proof of Theorem~\ref{lb-general}] Apply Lemma~\ref{is-double} with $k=\log N$, 
then apply Lemmas~\ref{solveA-lem} and~\ref{clusters-lem} to get
\begin{align*}
\probsub{\solve}{p} &\leq \probsub{\bigcup_{A\subset V:\abs{A} \in [\log N, 2\log N]} \solve_A}{p} \\
&\leq \sum_{v\in V}\sum_{k\in [\log N, 2\log N]} \sum_{A\in C(v,k)} \probsub{\solve_A}{p} \\
&\leq (N \log N ) \cdot \sup_{\alpha \in [1,2]} \left\{ (eD)^{\alpha \log N} \frac{2D}{\alpha\mu} N^{\alpha(1-\alpha\mu/(2D) - \log(\frac{2D}{\alpha\mu}))}  \right\}\\
&\leq \frac{2}{\mu} (\log N)\sup_{\alpha \in [1,2]} \exp\left[ \left(2\alpha+1 + \frac{\log D}{\log N} - \alpha^2 \mu/(2D) - \alpha\log\left(\frac{2}{\alpha\mu}\right) \right) \log N\right].
\end{align*}
When the $\limsup$ of the coefficient of $\log N$ in the exponential is strictly smaller than $0$ for any $\alpha \in [1,2]$, we see that $\probsub{\solve}{p}\to 0$.  Recalling that $\eta = \limsup (\log D / \log N)$, this condition is satisfied whenever  
\begin{align*}
\mu &< 2 e^{-2} \inf_{\alpha\in [1,2]} \frac{1}{\alpha} e^{-(1+\eta)/\alpha} \\
& = 2e^{-2} \min\left\{ e^{-(1+\eta)}, \frac{1}{2}e^{-(1+\eta)/2}\right\}.
\end{align*}
This completes the proof.
\end{proof}

\section{General graphs: upper bound}

\newcommand{\step}{\text{\tt Step}}

We first formulate a general theorem, then prove Theorem~\ref{intro-ub-list} 
in subsequent corollaries. This section is also devoted only to AE dynamics.

Fix a graph $G=(V,E)$, and positive integers $a$ and $k$. We will denote 
by $\cS_k$ a set of sequences of \df{length} $k$, consisting of vertices and \df{started at}
a fixed vertex $v_0\in V$. We will assume that $\cS_k$ is 
given recursively by a \df{building algorithm} as follows. 
Let $\cS_0=\{v_0\}$. For every $i\in [1,k]$, there exists a a \df{successor map} $\step_i$
defined on $\cS_{i-1}$ that attaches to every sequence $(v_0,\ldots, v_{i-1})\in \cS_{i-1}$ 
a set $\step_i(v_0,\ldots, v_{i-1})\subset V$, so that
$$
\cS_{i}=\{(v_0,\ldots, v_{i-1}, v_{i}): 
(v_0,\ldots, v_{i-1})\in \cS_{i-1}, v_{i}\in \step_i(v_0,\ldots, v_{i-1})\}.
$$
We also assume that each 
$B=\step_i(v_0,\ldots, v_{i-1})\subset V$ is ordered, and 
for $w\in B$, we let
$\overleftarrow B^w$ be the set of vertices in $B$ that are ahead of, or equal to, $w$
in the ordering. We think of $\overleftarrow B^{v_{i}}$
as the ``inspected'' vertices.
We call $\cS_k$ \df{$a$-admissible} if the following holds. Fix any sequence 
$(v_0,\ldots, v_k)\in \cS_k$, and let $B_i=\step_i(v_0,\ldots,v_{i-1})$, $1\le i\le k$, and 
$B_0=\{v_0\}$. Then, for $1\le i\le k$,
\begin{itemize}
\item $|B_i|\ge a$;
\item $\{v_0,\ldots, v_i\}$ is a connected subset of graph $G$; and
\item The selection up to $i$ does not affect selection at $i$, i.e.,
\begin{equation}\label{building-independence}
\left(\bigcup_{j=0}^{i-1}\overleftarrow B_j^{v_{j}}\right)
\bigcap
  B_{i}=\emptyset. 
\end{equation}
\end{itemize}

For a fixed probability $q$, we call $\size(G, v,q)$
the (random) number of vertices in the connected component of $v\in V$ in site percolation
on $G$ where vertices other than $v$ are open independently with probability $q$, and $v$ is open with probability $1$.

For a nondecreasing integer sequence $D_N$, 
we call a sequence of $\Gpuz$-graphs \df{$D_N$-regular} if the following is true for some 
constants $c,C>0$:  $D_N\ge 2C$ 
and there exist disjoint sets $V_\ell\subset V$, $\ell=1,\ldots, N^c$,  so that 
induced subgraphs $G_\ell=\Gpuz^{V_\ell}$ have the properties that 
\begin{itemize}
\item[(R1)] for each $\ell=1,\ldots, N^c$, $G_\ell$ contains a $cD_N$-admissible set of length at least $c\log N$ started at some $w_\ell\in V_\ell$; and
\item[(R2)]  
$\liminf_N\inf_\ell\prob{\size(G_\ell,w_\ell,C/D_N)\ge 2 D_N(\log N)^2}>0 $. 
\end{itemize}

\begin{theorem}\label{ub-general}
If $\Gpuz$ is {$D_N$-regular}, 
and $p=\frac{\mu}{D_N\log N}$ for a large enough constant 
$\mu$, then $\probsub{\solve}{p} \to 1$.
\end{theorem}


\begin{proof} 
We will use Lemma~\ref{dividing-up}, with three probabilities.
Assume first that $p_1=\mu_1/(D_N\log N)$, where $\mu_1=3/c^2$. 
Fix an $\ell$ and let $F_1$ be the event that 
$w_\ell$ is included in an internally solved set of size $c\log N$
within $G_\ell$.   By (R1), 
we may build such a cluster
by using the building algorithm for
the $cD_N$-admissible set of sequences. In this algorithm, we let $v_0=w_\ell$ and 
check the vertices in $\step_i(v_0,\ldots, v_{i-1})$ in their given order and stop
checking once we find one that is $\Gppl$-connected to $\{v_0,\ldots, v_{i-1}\}$.
Therefore,
\begin{equation}\label{ub-eq1}
\begin{aligned}
\probsub{F_1}{p_1}&\ge \prod_{i=1}^{c\log N} \left(1-(1-p_1)^{cD_Ni}\right)\\
&\ge \prod_{i=1}^\infty \left(1-e^{-p_1cD_Ni}\right)\\
&\ge \exp\left(\int_{0}^\infty\log(1-e^{-p_1cD_Nx})\, dx \right)\\
&= \exp\left(-\frac{\pi^2}{6c}\cdot \frac{1}{p_1D_N} \right)\\
&= N^{-\frac{\pi^2}{6c\mu_1}}.
\end{aligned}
\end{equation}

Now assume that $p_2=\mu_2/(D_N\log N)$, for $\mu_2=2C/c$.
Connect each pair of vertices with a {\it green\/} edge independently with probability $p_2$.
Then declare each vertex in $V_\ell$ to be {\it open\/} if it has a green edge to at least
one of the vertices in the largest internally solved subset of $V_\ell$ containing $w_\ell$
in the independent people graph with the $\Gppl$-edge probability $p_1$. If $F_1$ happens,
the probability that a vertex is open is, since $C/D_N\le 0.5$, 
at least $C/D_N$ independently of other vertices,
and (R2) applies.  Let $F_2$ be the event  that $w_\ell$ is included in an
internally solved set within $G_\ell$ of size $2D_N(\log N)^2$.
By (R2), (\ref{ub-eq1}) and Lemma~\ref{dividing-up},
\begin{equation}\label{ub-eq2}
\probsub{F_2}{p_1+p_2}\ge \alpha N^{-\frac{\pi^2}{6c\mu_1}},
\end{equation}
for some constant $\alpha>0$. Therefore, by (R1) and (\ref{ub-eq2}),
\begin{equation}\label{ub-eq3}
\begin{aligned}
&\probsub{\text{there is an internally solved set of size $2D_N(\log N)^2$}}{p_1+p_2}\\
&\ge 1- \left(1- \alpha N^{-\frac{\pi^2}{6c\mu_1}}\right)^{N^c}\\
&\ge 1-\exp(-\alpha N^{c-\frac{2}{c\mu_1}})\\
&=1-\exp(-\alpha N^{c/3}).
\end{aligned}
\end{equation}
Now let $p_3=1/(D_N\log N)$. 
If a fixed 
set $V_0$ of vertices has size at least $2D_N(\log N)^2$, then by 
Lemma~\ref{unstoppable-size}
\begin{equation}\label{ub-eq4}
\begin{aligned}
\probsub{\text{$V_0$ is unstoppable}}{p_3}
\ge 1-\frac 1N. 
\end{aligned}
\end{equation}
From Lemmas~\ref{unstoppable} and ~\ref{dividing-up}, and (\ref{ub-eq3}) and (\ref{ub-eq4}), it follows that 
$$
\probsub{\solve}{p_1+p_2+p_3}\ge \left(1-\frac 1N\right)\cdot \left(1-\exp(-N^{c/3})\right),
$$
and the result holds with $\mu\ge\mu_1+\mu_2+1$. 
\end{proof}

In a vertex-transitive graph, $D_N$ will typically be proportional to the degree. We now apply the
above theorem to some famous graphs.  In the corollaries that follow, note that $n$ is 
the natural parameter in the description of a family of graphs, and is
not equal to the total number of vertices.

\begin{cor} 
If $\Gpuz$ is the $d$-dimensional lattice torus with $V=\bZ_n^d$,
there exists a universal constant $C$ so that $p\ge C/(d^2\log n)$ 
implies $\probsub{\solve}{p} \to 1$.
\end{cor}

\begin{proof} In this, and subsequent, proofs we will omit the obvious integer parts
required to make certain quantities integers.  
In the torus, find $n^{d/2}$ disjoint subcubes 
congruent to $[1,\sqrt n]^d$. In each of these subcubes, consider the set 
of oriented percolation paths, which is clearly $d$-admissible: 
$B_i$ only depends on $v_{i-1}$ and is the set
$\{v_{i-1}+e_1,\ldots, v_{i-1}+e_d\}$ (where $e_j$ are the standard basis vectors).
The order is immaterial, as (\ref{building-independence}) holds
with all $\overleftarrow B_j^{v_{j}}$ replaced by $B_j$.
Then (R1) holds with $c=1$, provided $\sqrt n\ge 2d\log n$. 
To verify (R2), 
use the well-known fact that the critical 
probability of site percolation on $\bZ^d$ scales as $1/(2d)$ \cite{Kes2}. 
\end{proof}

\begin{cor} \label{Cor:Z2 long range}
If $\Gpuz$ is the graph with vertices $V=\bZ_n^2$ and edges between all pairs of vertices $x$ and $y$ such that $||x-y||_\infty\le r$, there exists an universal constant $C$ so that 
$P\ge C/(r^2\log n)$ implies 
$\probsub{\solve}{p} \to 1$.
\end{cor}

\begin{proof}
Divide $\bZ_n^2$ into $\sqrt n\times \sqrt n$ squares. In each, consider the set 
of neighborhood paths, which start at the lower left corner and are oriented 
(i.e, both coordinates are increasing along the paths). 
Here, $B_i$ is the $(r+1)\times (r+1)$ square with its leftmost lowest corner
at $v_{i-1}$, with $v_{i-1}$ excluded. Moreover, the ordering of points in $B_i$ is given as follows:
$(x_1,y_1)<(x_2,y_2)$ if
either $x_1+y_1<x_2+y_2$; or $x_1+y_1=x_2+y_2$ and $x_1<x_2$. Then (R1) holds 
provided $\sqrt n>2r\log n$. 
See \cite{Gra} for the relevant 
site percolation result to verify (R2).
\end{proof}

\begin{cor}
If $\Gpuz$ is the $n$-dimensional hypercube with $V=\{0,1\}^n$, 
there exists an universal constant $C$ so that 
$p\ge C/n^2$ implies  $\probsub{\solve}{p} \to 1$.
\end{cor}

\begin{proof}
Let $d$ be the Hamming distance, and divide the graph into $2^{n/4}$ disjoint $(3n/4)$-dimensional subcubes.
In each subcube we find a $(n/4)$-admissible set of length $n/4$ by letting $B_i$ be the set of hypercube-neighbors $w$ of $v_{i-1}$ that have
$d(w,v_{0})>d(v_{i-1},v_0)$, and the order is immaterial.
To verify (R2), use the percolation result from \cite{BKL}.  
\end{proof}

To prove our Hamming torus result in low dimensions, we need a lemma on connectivity 
of high-density random subsets. 

\begin{lemma}\label{hamming-2d-connectivity}
Assume that every vertex of the two-dimensional Hamming torus 
with vertex set $V=[0,n-1]^2$ is open independently with 
probability that may vary among vertices but is bounded below by 
$n^{-\gamma}$ for some $\gamma<2/3$. Then, with probability
approaching $1$, for each pair $x,y\in V$ there exist open vertices
$z_1,z_2,z_3$ so that $z_1$ is a neighbor of $x$ and of $z_3$, and 
$z_2$ is a neighbor of $y$ and of $z_3$. Furthermore, 
a.~a.~s.~all open vertices form a connected set of size 
at least $0.5n^{2-\gamma}$. 
\end{lemma}

\begin{proof} Fix any two vertices $x$ and $y$, and let $E$ be the event that
vertices $z_1,z_2,z_3$ with specified properties exist.
Let $E_1$ be the event that the
horizontal line through $x$ and the vertical line through $y$ both have 
at least $0.5 n^{1-\gamma}$ open vertices. Then, by Lemma~\ref{binomial-ld}, 
$$
\prob{E_1}\ge 1-2\exp(-n^{1-\gamma}/7),
$$  
Conditioned on $E_1$, there are at least $0.25 n^{2-2\gamma}$ independent candidates 
for an open vertex that is incident to open vertices in both neighborhoods of 
$x$ and $y$. As $\gamma<2/3$, by Lemma~\ref{binomial-ld}, 
$$
\prob{E|E_1}\ge 1-\exp(-0.03\cdot n^{2-3\gamma}),
$$
which easily finishes the proof of the first claim. The second claim is 
then another easy application of Lemma~\ref{binomial-ld}. 
\end{proof}

\begin{cor}
If $\Gpuz$ is the $d$-dimensional Hamming torus on the vertex set $\bZ_n^d$,
there exists a universal constant $C$ so that 
$p\ge C/(d^2n\log n)$ implies $\probsub{\solve}{p} \to 1$.
\end{cor}

\begin{proof}
Assume first that $d\ge 4$. Let $d_1=\lfloor d/2\rfloor - 1$. For any $d_1$-tuple
$a=(a_1,\ldots a_{d_1})$, let $M_a$ be the set of vertices whose last
$d_1$ coordinates equal $a$. There are $n^{d_1}$ disjoint sets $M_a$, 
each of which is a Hamming torus of dimension $d-d_1\ge 3$. In each of these tori, 
$\step_i(v_0,\ldots, v_{i-1})$ 
comprises vertices that are in the neighborhood of $v_{i-1}$, but not in the 
neighborhood of any of the previous points, $v_0, \ldots, v_{i-2}$. This defines a $dn/4$-admissible set of length $d\log n$;
the order is again immaterial. This verifies (R1). Theorem 1.2 
from \cite{Siv} implies that the giant component in $M_a$ is on the order of $n^{d-d_1-1} \gg D_N (\log N)^2 = d^3 n (\log n)^2$, which implies (R2).

Theorem~\ref{ub-general} thus handles the case $d\ge 4$. The cases 
$d=2, 3$ require a modified argument that we now present.  For $d=2$ consider the entire puzzle graph, and for $d=3$ consider a fixed two-dimensional subgraph.  Assume first that  
the $\Gppl$ probability is $p_1=\mu_1/(n\log n)$. 
We will 
describe a sequence $(z_i)$ of vertices, divided into {\it ordinary\/} and 
{\it base\/} vertices. 
The sequence starts with an arbitrary $z_0\in V$, a base vertex. 
Given vertices $z_0,\ldots, z_{i-1}$, let $z_j$, be the base vertex 
with the largest index $j<i$. Inspect one by one all vertices in the neighborhood of
$z_{i-1}$ which are not in the neighborhood of any previous vertices, $z_0, \ldots, z_{i-2}$, 
until either: 
\begin{itemize}
\item a vertex that is $\Gppl$-connected to one of
the vertices $z_j,\ldots, z_{i-1}$ is found, which is then declared an ordinary vertex 
$z_i$; or 
\item all vertices are exhausted, in which case $z_i$ is a new base vertex, 
selected arbitrarily
outside of the neighborhoods of vertices $z_0,\ldots, z_{i-1}$.
\end{itemize} 
We continue this construction until we either: encounter a subsequence of $n^{0.7}$ 
consecutive ordinary vertices, 
in which case we call the sequence {\it successful\/}; or the sequence reaches 
length $n^{0.9}$. By construction, the number of vertices available 
for inspection is always at least $n-o(n)$. Thus, conditioned on any outcome of prior inspections,  
a new base vertex is
the last base vertex with probability at least $n^{-0.1}$  for a large enough 
$\mu_1$, by a calculation similar to (\ref{ub-eq1}). The number of base vertices 
in an unsuccessful sequence is at least $n^{0.2}$ and so 
\begin{equation}\label{hamming-eq1}
\begin{aligned}
&\probsub{\text{there is an internally solved set of size $n^{0.7}$}}{p_1}\\
&\ge 
\probsub{\text{sequence successful}}{p_1}\ge 1- \exp(-n^{0.1}).
\end{aligned} 
\end{equation}
Now let $p_2=1/(n\log n)$, and as in the proof of Theorem~\ref{ub-general} 
assume that each pair of vertices is 
connected by a green edge with probability $p_2$, and then declare a vertex open if it 
is connected by a green edge to the largest internally solved set at edge density $p_1$. 
On the event that the largest $p_1$-internally solved cluster has size at least $n^{0.7}$, the $p_2$-probability of a fixed vertex being open is at least 
$$
\left(1-(1-p_2)^{n^{0.7}}\right)>n^{-1/2},
$$
independently of other vertices. By (\ref{hamming-eq1}) and
Lemma~\ref{hamming-2d-connectivity}, 
$$
\probsub{\text{there is an internally solved set of size $0.5n^{3/2}$}}{p_1+p_2} \to 1.
$$
As $0.5 n^{3/2}\gg n(\log n)^2$, the proofs for both $d=2$ and $d=3$ are easily concluded using the unstoppability Lemma~\ref{unstoppable-size}
and additional density $p_3 = 1/(n \log n)$,
as at the end of the proof of Theorem~\ref{ub-general}.
\end{proof}

As we see from the above examples, for many vertex-transitive graphs of $N$ vertices and degree
$D$, $p_c$ scales as $1/(D\log N)$. This is however not always true. 
The easiest counterexample is the 
complete graph $K_n$ where $p=\mu/(n\log n)$ yields a disconnected 
graph $\Gppl$, and in fact $p_c\sim \log n/n$, as observed in \cite{BCDS}. We now 
show by an example that 
this scaling may fail to hold even if
$\Gppl$ is connected. 

\begin{prop}\label{ub-counterexample}
Consider the Cartesian product graph $\Gpuz=K_{n}\times R_{(\log n)^3}$ of a complete
graph $K_n$ (of $n$ vertices) and a ring graph $R_{(\log n)^3}$
(of $(\log n)^3$ vertices), and
$p=\mu/(n\log n)$ for some constant $\mu>0$. Then $\Gppl$ is a.~a.~s.~connected, but
$\probsub{\solve}{p}\to 0$.
\end{prop}

\begin{proof}
The first statement is clear as the threshold for $\Gppl$-connectivity scales
as $1/(n(\log n)^2)$. To prove the second statement, we find an upper bound
for the number $C(v,k)$ of
connected sets of size $k=\cO(\log n)$ that include a specific vertex $v$.

Divide the set of vertices into copies of $K_n$, denoted by $K'_i$, $i=1,\ldots, (\log n)^3$,
which are in cyclic order connected by the ring graph edges.
We will assume that the vertices in $K_i'$ have a prescribed order.
A connected set $A$, with $v\in A$, of size $k$ in $\Gpuz$ must be divisible into $\ell$
contiguous sets (on the ring) $K_{i_0+1}', \ldots, K_{i_0+\ell}'$, for some $\ell\in [1,k]$
and some $i_0$. Thus there exist $k_1,\ldots, k_\ell\ge 1$,
with $k_1+\ldots+k_\ell=k$, so that there are $k_i$ vertices in each
$A\cap K_{i_0+i}'$, $i=1,\ldots, \ell$.
We now fix a choice of $A$ recursively as follows.
Once the $k_i$ points in $K_{i_0+i}$ are chosen, we choose the first point
in the ordering that has a ring connection to a point in $A\cap K_{i_0+i+1}$.
This fixes one of $k_{i+1}$ points in $A\cap K_{i_0+i+1}$, which we call
the {\it base point\/},
and we have at most $\binom n {k_{i+1}-1}$ choices for the others.
We choose, say, the first point in the set $A\cap K_{i_0+1}$ as its base point.
This gives
\begin{equation}\label{ub-counterexample-eq1}
\begin{aligned}
C(v,k)&\le kn\sum_\ell\sum_{k_1,\ldots, k_\ell} k_1\cdots k_{\ell-1}
\binom n {k_{1}-1}\cdots \binom n {k_{\ell}-1}\\&
\le n \sum_{\ell=1}^k k^{\ell} \binom {n\ell} {k-\ell}
\le n\cdot\max_{1\le \ell\le k} k^{\ell+1} \binom{n\ell}{k-\ell}.
\end{aligned}
\end{equation}
We now use $\binom{n}{k}\le (en/k)^k$, so that from (\ref{ub-counterexample-eq1})
$$
C(v,k)\le n\exp((\ell+1)\log k+(k-\ell)(1+\log n+\log \ell-\log(k-\ell))),
$$
and the derivative of the expression inside $\exp$ with respect to $\ell$
is
$$
\log k+\log\left(\frac k\ell-1\right)+\frac k\ell -3 -\log n.
$$
This last expression is negative for all $n\geq e^2$, $\ell\ge 1$, and 
$1\leq k\le \frac 12\log n$.
Assuming this,
$$
C(v,k)\le \exp(k\log n-k\log k+\cO(\log n)).
$$
Now by Lemma~\ref{solveA-lem}, for an $A$ of size $|A|=\alpha\log n$,
$$
\probsub{\solve_A}{p}\le \exp\left(-\alpha (\log n)^2 +\cO(\log n)\right).
$$
When $\solve$ occurs, so does $\solve_A$ for some $A$ with
$A\in[\frac 14\log n, \frac 12\log n]$,
and then
$$
\probsub{\solve}{p}\le \exp\left(-\frac 14\log n\log\log n +\cO(\log n)\right),
$$
which ends the proof.
\end{proof}

\section{Ring puzzle: sharp transition}

In this section we assume that $\Gpuz$ is the ring graph $\bZ_n$ of $n$ vertices, that
$\tau=1$, $\theta=\infty$ and $\sigma\ge 1$ is arbitrary, and prove 
Theorems~\ref{intro-ring-thm} and~\ref{intro-final-time-thm}. 

\begin{lemma}\label{g-properties}
The function $g_\sigma$ is positive, decreasing, and convex on $(0,\infty)$. 
\end{lemma}
\begin{proof}
Positivity is obvious, and 
$$
g_\sigma'(x)=-\frac 1 {(\sigma-1)!\sum_{i=\sigma}^\infty \frac{x^{i-\sigma+1}}{i!}}.
$$
implies the other two properties.
\end{proof}

\begin{lemma} \label{D-bound}
Fix $a,b,\epsilon>0$. Then there exists a $\delta>0$ so that 
the following holds. 
Assume $R\subset R'$ are intervals with 
$|R|=x/p$ and $|R'|=(x+\delta)/p$. Then, if $p$ is small enough, 
$$
p\log \probsub{D(R,R')}{p}\le -(1-\epsilon)g_\sigma(x)\delta, 
$$
for all $x\in [a,b]$.  
\end{lemma}

\begin{proof}
Let $M$ be the 
number of vertices in $R'\setminus R$ that have no $\Gppl$-neighbors 
in $R'\setminus R$. Let $L=|R'|-|R|$. We will show that $M$ is very likely to be close 
to $L$ even in the large deviation regime. 
 
Let $X$ be the number of $\Gppl$-edges between vertices in $R'\setminus R$. Then, 
by (\ref{binomial-easy}), 
$$
\probsub{L-M \ge\epsilon L}{p}\le 
\probsub{X\ge \epsilon L/2}{p}\le \prob{\Bin(L^2/2, p)\ge \epsilon L/2}
\le 
\left(\frac{3\delta}{\epsilon}\right)^{\epsilon L/2}
$$
Further, for $p$ small enough, by Poisson approximation \cite{BHJ}, for any $y\in R'\setminus R$, 
$$
\probsub{y\text{ has at least $\sigma$ $\Gppl$-neighbors in $R$}}{p}\le e^{-g_\sigma(x)}+p.
$$
Therefore, by independence between edges within $R'\setminus R$ and those 
leading out of this set,
\begin{equation}
\begin{aligned}
\probsub{D(R,R')}{p}&\le (e^{-g_\sigma(x)}+p)^{(1-\epsilon)L}+\probsub{M<(1-\epsilon) L}{p}\\
&\le \exp(-g_\sigma(x)(1-\epsilon)L+Cp(1-\epsilon)L)+2\cdot \exp\left[-\left(\frac 12 
\log\frac{\epsilon}{3\delta}\right)\epsilon L\right]
\end{aligned}
\end{equation}
Here, $C$ is a constant that depends only on $a$ and $b$. The second term 
can be made smaller than the first term (uniformly for $x\in [a,b]$), 
by choosing $\delta$ small enough, 
and then the result follows.  
\end{proof}

\begin{lemma}\label{seed-bound} Fix any interval $R$,
$$
\probsub{R\text{ \rm internally solved}}{p}\le (2|R|p)^{|R|}.
$$
\end{lemma}

\begin{proof}
This follows from (\ref{binomial-easy}) as 
$\probsub{|\Epeople^R|\ge |R|}{p}=\prob{\Bin(|R|(|R|-1)/2, p)\ge |R|}$.
\end{proof}

\begin{lemma}\label{solved-subintervals}
Let $R$ be an internally solved interval with $|R|\ge 2$. Then there are non-empty 
internally solved intervals $R'$, $R''$ which partition $R$ such that $\langle\{R', R''\}\rangle=R$. 
\end{lemma}

\begin{proof}
This follows from the slowed-down jigsaw percolation on the pair $\Gpuz^R$, $\Gppl^R$, 
whereby $\cG_t'$ in (J6) 
has exactly one edge. If $T_f^R$ is the minimal time when the final configuration is reached, 
then the partition at time $T_f^R-1$ satisfies the theorem.
\end{proof}

We now adapt the key concepts from \cite{Hol} that we use 
to prove the lower bound on $p_c$. None of what we do in the 
next two lemmas is original, but we 
give some details mainly to demonstrate how much simpler the argument is 
in this one-dimensional case. 

Pick small constants $T$ and $Z$; we will also assume that $T$ is much smaller than $Z$. 
A \df{hierarchy} $\cH$ is a finite directed tree in which each vertex $u$ is associated 
with a nonempty interval $R_u$. 
A special vertex $r$, the {\it root\/}, is associated with an interval $R$. 
All edges point away from the root, and $u\to v$ implies 
$R_u\supset R_v$. Each vertex $u$ is one 
of the three kinds:
\begin{itemize}
\item a {\it seed\/} with no children;
\item {\it normal\/} with a single child $v$, written as $u\Rightarrow v$; or
\item a {\it splitter\/} with two children $v,w$, written as $u\rightrightarrows v,w$. 
\end{itemize}
To say that a \df{hierarchy occurs} we further require that $R_v$ and $R_w$ partition $R_u$ whenever 
$u\rightrightarrows v,w$, that $R_u$ is internally solved for each seed $u$ and
that $D(R_v, R_u)$ happens whenever $u\Rightarrow v$. Finally, we impose the 
following conditions on the lengths of the intervals:
\begin{itemize}
\item[(H1)] $|R_u|<2Z/p$ for every seed $u$; 
\item[(H2)] $|R_u|\ge 2Z/p$ for every splitter and every normal vertex;
\item[(H3)] $u\Rightarrow v$ and $v$ is not a splitter implies $|R_u|-|R_v|\in [T/(2p), T/p]$;
\item[(H4)] $u\Rightarrow v$ and $v$ is a splitter implies $|R_u|-|R_v|\le T/p$; and
\item[(H5)] $u\rightrightarrows v,w$ implies $|R_u|-|R_v|\ge T/(2p)$ and $|R_u|-|R_w|\ge T/(2p)$.
\end{itemize}

\begin{lemma}\label{hierarchy-occurs}
If $R$ is an internally spanned interval, a hierarchy with $R_r=R$ occurs.  
\end{lemma}
\begin{proof}
See the proof of Proposition 32 in \cite{Hol}.
\end{proof}

\begin{lemma} \label{ring-upper-bound}
Fix $a,\epsilon>0$, $b\ge a$, and an interval $R$ of length $\lceil b/p\rceil$. Then 
$$
\probsub{\solve_R}{p}\le \exp\left(-p^{-1}(1-2\epsilon)\int_a^b g_\sigma(x)\,dx\right),
$$
for small enough $p$.
\end{lemma}

\begin{proof} For an interval $R'$, we let 
$$V(R')=-p|R'|\log (2p|R'|)$$ 
and for $R'\subset R''$, we let 
$$U(R',R'')=\int_{p|R'|}^{p|R''|} g_\sigma(x)\,dx.$$
Observe that $U(R', R'')\le g(p|R'|)(p|R''|-p|R'|)$ by Lemma~\ref{g-properties}.  

Then we have, for a given hierarchy $\cH$, 
\begin{equation}\label{ringpf-eq1}
\begin{aligned}
\probsub{\text{$\cH$ occurs}}{p}&\le 
\exp\left(-p^{-1}\left[(1-\epsilon)\sum_{u\Rightarrow v}U(R_u, R_v)+\sum_{w \text{ seed}}V(R_w)\right]
\right)
\end{aligned}
\end{equation}

Here, the first sum is over pairs $u,v$ of vertices in $\cH$ with 
$u\Rightarrow v$ and the second sum is over seeds $w$ of $\cH$.  
We get (\ref{ringpf-eq1}) from Lemmas~\ref{D-bound} (with a suitable choice of 
$T$, which is now fixed) and \ref{seed-bound}. 

Now let $S$ be the interval with its length equal to the combined 
length of all seeds in $\cH$, positioned inside $R$ (say, so that 
the left endpoints agree, although the exact position is not essential). Note that $S$ is analogous to what~\cite{Hol} refers to as a {\em pod}.
Then we claim that 
\begin{equation}\label{pod}
\sum_{u\Rightarrow v}U(R_v, R_u)\ge U(S,R)
\end{equation}
This assertion is proved by induction on the number of vertices in $\cH$. 
If $R$ is a seed, then (\ref{pod}) is trivial. If the root $r$ (with $R_r=R$) 
is a normal vertex with child $y$, then apply the induction hypothesis to 
the hierarchy rooted at $y$ to get 
\begin{equation*}
\sum_{u\Rightarrow v}U(R_v, R_u)\ge U(R_y, R)+U(S,R_y)\ge U(S,R).
\end{equation*}
If $r$ is a splitter with children $y_1$, $y_2$, then we apply the induction 
hypothesis to the hierarchies rooted at $y_1$ and $y_2$ with respective pods 
$S_1$ and $S_2$, to get
\begin{equation*}
\sum_{u\Rightarrow v}U(R_v, R_u)\ge U(S_1, R_{y_1})+U(S_2,R_{y_2})\ge U(S,R), 
\end{equation*}
as it is easy to see by Lemma~\ref{g-properties} using $|S|=|S_1|+|S_2|$ and 
$|R|=|R_{y_1}|+|R_{y_2}|$. This establishes (\ref{pod}). 

For a seed $w$, by the definition of $V$ and property (H1) of seeds, 
$$
\sum_{w \text{ seed}}V(R_w)\ge p|S|\log\frac 1{4Z}.
$$
Therefore, 
\begin{equation}\label{ringpf-eq2}
\begin{aligned}
\probsub{\text{$\cH$ occurs}}{p}&\le 
\exp\left(-p^{-1}\left[(1-\epsilon)U(S,R)+ p|S|\log\frac 1{4Z}\right]\right).
\end{aligned}
\end{equation}
Choose $Z$ so small that $\log\frac 1{4Z}>\lambda_c/a$. 
Then, if $p|S|>a$, 
$$
p|S|\log\frac 1{4Z}\ge (1-\epsilon)\int_a^b g_\sigma(x)\,dx. 
$$
If $p|S|\le a$, then $U(S,R)\ge \int_a^b g_\sigma(x)\,dx$. 
Then, by (\ref{ringpf-eq2}), 
\begin{equation}\label{ringpf-eq3}
\begin{aligned}
\probsub{\text{$\cH$ occurs}}{p}&\le 
\exp\left(-p^{-1}(1-\epsilon)\int_a^b g_\sigma(x)\,dx\right).
\end{aligned}
\end{equation}

It is not hard to see \cite{Hol} that 
the number of hierarchies $\cH$ that satisfy (H1--5) is bounded above 
by $p^{-K}$, where $K$ is a constant that depends only on $b$ and $T$. 
But if $p$ is small enough, 
$$K\log\frac 1p\le \epsilon p^{-1}\int_a^b g_\sigma(x)\,dx, 
$$
which ends the proof.
\end{proof}

\begin{proof}[Proof of Theorem~\ref{intro-ring-thm}]
Assume that $p=\lambda/\log n$ for some $\lambda<\lambda'<\lambda_c$. 
Then choose $a$, $b$ and $\epsilon$ so that $(1-2\epsilon)\int_a^b g_\sigma(x)\,dx>\lambda'$. 
By Lemmas~\ref{is-double} and~\ref{ring-upper-bound},   
\begin{equation*}
\begin{aligned}
\probsub{\solve}{p}&\le \probsub{
\bigcup_{R\, :\, |R|\in [b/p, 2b/p]}\solve_{R}}{p}\\
&\le b \lambda n\log n\cdot \exp(-(\lambda'/\lambda)\log n)=b\lambda n^{1-\lambda'/\lambda}\log n\to 0, 
\end{aligned}
\end{equation*}
which proves the lower bound for $p_c$. 

The proof of the upper 
bound generalizes the one in \cite{BCDS} for $\sigma=1$. Fix a small $\epsilon>0$. We will later 
choose a small $a>0$ and a large $b$ dependent on $\epsilon$.
 
Let $p=\lambda/\log n$, for some $\lambda>\lambda_c+5\epsilon$.
Fix an interval $R$ with length $L=\lceil 3\epsilon^{-1}(\log n)^2\rceil$. 
We will find a lower bound for $\probsub{\solve_R}{p}$. For 
notational convenience, we will assume the left endpoint of $R$ 
is at the origin.

Let 
$$H_k=\{k\text{ is $\Gppl$-connected to at least $\sigma$ points in $[0,k-1]$}\},$$
and define the following four events 
\begin{equation*}
\begin{aligned}
&G_1=\{\{k,k+1\}\in \Gppl, \text{ for all }k\le p^{-1/2}\},\\
&G_2=\cap_{p^{-1/2}<k\le ap^{-1}}H_k, \\
&G_3=\cap_{ap^{-1}<k\le bp^{-1}}H_k, \\
&G_4=\cap_{bp^{-1}<k<L}H_k. 
\end{aligned}
\end{equation*}
Clearly these are independent events and $G_1\cap G_2\cap G_3\cap G_4\subset \solve_R$. 
We now estimate their probabilities. Clearly, 
\begin{equation}\label{G1-est}
\probsub{G_1}{p}\ge \exp\left(-p^{-1/2}\log \textstyle{\frac 1p}\right)>\exp(-p^{-1}\epsilon), 
\end{equation}
for small enough $p$. 
Moreover, with all products and sums over $k$ in the corresponding range 
\begin{equation}\label{G2-est}
\begin{aligned}
\probsub{G_2}{p}&\ge \prod_k \prob{\Bin(k-1,p)\ge \sigma}\\
&\ge \prod_k \prob{\Bin(\lfloor (k-1)/\sigma\rfloor,p)\ge 1}^\sigma \\
&\ge \exp\left(\sigma\sum_k\log(1-e^{-pk/(2\sigma)})\right)\\
&\ge \exp\left(p^{-1}\sigma \int_0^a\log(1-e^{-x/(2\sigma)})\, dx\right)\\
&\ge \exp(-p^{-1}\epsilon), 
\end{aligned}
\end{equation}
for small enough $a$. Similarly, 
\begin{equation}\label{G4-est}
\begin{aligned}
\probsub{G_4}{p}\ge
 \exp\left(p^{-1}\sigma \int_b^\infty\log(1-e^{-x/(2\sigma)})\, dx\right)\ge \exp(-p^{-1}\epsilon ),
\end{aligned}
\end{equation}
for large enough $b$. 
Finally, for $p$ small enough \cite{BHJ}, 
\begin{equation}\label{G3-est}
\begin{aligned}
\probsub{G_3}{p}&\ge \prod_k \left(\prob{\text{Poisson}(kp)\ge \sigma}-2p\right)\\
&\ge \exp\left(p^{-1}\int_a^b\log(\prob{\text{Poisson}(x)\ge \sigma}-2p)\, dx\right)\\
&\ge \exp\left(-p^{-1}\epsilon-p^{-1}\int_a^b g_\sigma(x)\, dx\right)\\
&\ge \exp\left(-p^{-1}\lambda_\sigma-p^{-1}\epsilon\right).
\end{aligned}
\end{equation} 
From (\ref{G1-est}--\ref{G3-est}) we get 
\begin{equation}\label{solve-R}
\probsub{\solve_R}{p}\ge \exp\left(-p^{-1}\lambda_\sigma-4p^{-1}\epsilon\right)
\end{equation}
and then 
\begin{equation}\label{is-exists}
\begin{aligned}
&\probsub{\text{exists an internally solved interval of length $L$}}{p}\\
&\ge 1-\left(1-n^{-(\lambda_\sigma+4\epsilon)/(\lambda_\sigma+5\epsilon)}
\right)^{\frac 14\epsilon n (\log n)^{-2}}\to 1.
\end{aligned}
\end{equation}
For a fixed interval $R$ with length $L$, we also have, 
\begin{equation}\label{is-unstoppable}
\begin{aligned}
\probsub{\text{$R$ is unstoppable}}{\epsilon/\log n} 
\ge 1-n^{-2},
\end{aligned}
\end{equation}
by Lemma~\ref{unstoppable-size}.
For the final step, assume that $p=\lambda/\log n$ with $\lambda>\lambda_\sigma+6\epsilon$. By 
Lemmas~\ref{unstoppable} and~\ref{dividing-up}, 
(\ref{is-exists}), and (\ref{is-unstoppable}), $\probsub{\solve}{p}\to 1$. This finishes
the proof.
\end{proof}

To prove Theorem~\ref{intro-final-time-thm}, we need a simple observation. 

\begin{lemma}\label{Tf-lemma}
Assume $\solve$ happens and there is an interval $R\subset V$ with $|R|=L$ such that 
no element of the partition $\final_R$ exceeds size $\ell$. Then $T_f\ge L/(2\ell)$. 
\end{lemma}

\begin{proof}
By monotonicity of jigsaw percolation, we may start with the partition 
$\cR^0=\final_R\cup\{R^c\}$. 
If $t<L/(2\ell)$ the graph $\cG_t$ has at most two edges, as the cluster that 
contains $R^c$ may only advance into $R$ from either end, and so $t<T_f$. 
\end{proof}

\begin{proof}[Proof of Theorem~\ref{intro-final-time-thm}]
The result for $\lambda<\lambda_\sigma$ is proved in the previous theorem, as
all intervals in the final partition are a.~a.~s.~of logarithmic size. 
We assume that $\lambda>\lambda_\sigma$ for the rest of the proof. 

Take any $\lambda'<\lambda_\sigma$, and fix an interval $R$ of length $e^{\lambda'/p}$. 
Then a.~a.~s.~$\final_R$ only contains intervals of size $C\log n$ for some 
constant $C$. By Lemma~\ref{Tf-lemma}, 
$$T_f\ge (2C\log n)^{-1} e^{\lambda'/p}=(2C\log n)^{-1}n^{\lambda'/\lambda}, 
$$
and so 
\begin{equation}\label{Tf-eq1}
\liminf (\log n)^{-1}\log T_f\ge \lambda'/\lambda.
\end{equation}
 
Now take $\lambda''\in (\lambda_\sigma, \lambda)$. 
For any fixed  interval $R$ of length $L=e^{\lambda''/p}$, and for  
$p$ small enough, $\solve_R$ happens with probability at least $0.5$. 
Futhermore, a.~a.~s., every interval of length $L$ is unstoppable (as is every interval 
of length $\gg (\log n)^2$; see the previous proof). Therefore, by the 
run-length problem, a.~a.~s.~each $v\in V$ is at most $C\log n$ intervals away from 
an internally solved unstoppable interval of length $L$, for some constant $C$. 
It follows that 
$
T_f\le L \cdot C\log n,
$
and 
\begin{equation}\label{Tf-eq2}
\limsup (\log n)^{-1}\log T_f\le \lambda''/\lambda.
\end{equation}
The two inequalities (\ref{Tf-eq1}) and (\ref{Tf-eq2}) end the proof.
\end{proof}

\section{Bounded degree graphs with large $\sigma$}

In this section we prove Theorem~\ref{intro-large-sigma-thm}. 
Thus, our dynamic parameters are $\tau=1$, $\theta=\infty$, and 
a large $\sigma$. We also assume that the maximum degree of $\Gpuz$ with $|V|=N$ is bounded 
above by a fixed constant $D$. All constants will depend on $D$, in addition to 
explicitly stated dependencies. 

We will need the method used to prove Theorem 2 of \cite{BCDS}. The essence of 
this method is presented in the lemma below, whose simple proof we provide for completeness. 
Let $T$ be a tree with $N$ vertices and $N-1$ edges; generate $2(N-1)$ oriented edges
by giving each edge both orientations. 
Consider an oriented cycle of length $\ell$, i.e., a vector of oriented edges  
$(f_0,\ldots, f_{\ell-1})$ such that the head-vertex of $f_i$ is the tail-vertex of 
$f_{i+1}$, $0\le i\le \ell$; in a cycle, indices are always reduced modulo $\ell$. A {\it segment\/}
of length $m\le \ell$ in such cycle is a vector $(f_i,\ldots, f_{i+m-1})$, for some $i$.  
For a segment, we call its {\it edge set\/} and {\it vertex set\/} the set of 
all its (unoriented) edges, and the set of all vertices incident to its edges, respectively.

\begin{lemma}\label{tree-cover} If $T$ is a tree with $N$ vertices, 
there is an oriented cycle that includes each oriented edge exactly once.
Further, for any integer $L\in [1,N-1]$ 
there exist $\lceil (N-1)/(2L^2)\rceil$ segments with the following 
three properties: (1) the edge set of each segment has 
cardinality $L$; (2) any two segments have
disjoint edge sets; and (3) any two segments have 
vertex sets whose intersection is at most a singleton.   
\end{lemma}

\begin{proof}
The first statement is well-known and easy to prove by induction.  Observe that 
it implies that the vertex and edge sets of any segment determine a connected subtree of $T$. 
This observation, together with (2), implies (3). 
 
Start with any segment with edge set of size $L$. 
This segment has length at most $2L$. Assume $j$ segments
satisfying (1) and (2) are found. A $(j+1)$st segment can then be selected provided 
that there is an segment of length $2L$ whose edge set is disjoint from the union $U$ of edge 
sets of all $j$ segments. The set $U$ has size $Lj$, and these edges are in at 
most $4L^2j$ segments of length $2L$. One of $2(N-1)$ segments of length $2L$ 
therefore contains no edge in $U$ provided $2(N-1)>4L^2j$, that is, $j<(N-1)/(2L^2)$. 
\end{proof}

\begin{proof}[Proof of Theorem~\ref{intro-large-sigma-thm}]
To prove the upper bound (which does not require the bound on the degree), 
we prove that 
\begin{equation}\label{large-sigma-eq0}
\limsup_{N\to\infty} p_c\log N \le \lambda_\sigma.
\end{equation}
To prove (\ref{large-sigma-eq0}), replace $\Gpuz$ by its spanning tree $T$.
As in the proof of Theorem~\ref{intro-ring-thm} in Section 5, 
assume that $\lambda>\lambda_\sigma+5\epsilon$ and $p=\lambda/\log N$.
Then use Lemma~\ref{tree-cover} with $L=\lceil 3\epsilon^{-1} (\log N)^2\rceil$.
Let $T_i$ be subtrees given by the edge and vertex sets of the resulting segments.
As each $T_i$ is connected, it is easy to adapt the proof of (\ref{solve-R}) to get the same lower bound on the probability that $T_i$ is internally solved. Notice also that, 
by (3), the subtrees $T_i$ are internally solved independently. Thus, the probability that at least
one of them is internally solved is bounded below by the following analogue of (\ref{is-exists})
$$
1-\left(1-N^{-(\lambda_\sigma+4\epsilon)/(\lambda_\sigma+5\epsilon)}
\right)^{(N-1)/(2L^2)}\to 1, 
$$
since $(N-1)/(2L^2)\sim 4.5\epsilon^2 N/(\log N)^4$. The proof of (\ref{large-sigma-eq0})
is now finished in the same way as the proof of Theorem~\ref{intro-ring-thm} after (\ref{is-exists}). 
Once we have (\ref{large-sigma-eq0}), the upper bound follows from the 
following asymptotic fact, valid as $\sigma\to\infty$, 
$$
\lambda_\sigma\sim \int_0^\sigma (\sigma\log\sigma-\sigma\log x -\sigma+x)\,dx =\frac 12\sigma^2, 
$$
which follows from an elementary large deviation estimate for Poisson distribution. 
 
To prove the lower bound, assume that $p=\lambda/\log N$. 
Fix a $\Gpuz$-connected set $A\subset V$ of size $|A|=\ell$, with 
$\log N/\sigma\le \ell\le 2\log N/\sigma$.
Consider the version of slowed-down jigsaw percolation on 
$(\Gpuz^A, \Gppl^A)$ in which $\cG_t'$ in (J6) has 
at most one edge when $t\ge 1$. When $t=0$, we let $\cG_0'=\cG_0$, that is, all components 
of doubly connected vertices are merged at the first time step. 

For a vertex $v\in A$, let 
$$
E_v=\{\text{$\Gppl^A$-degree of $v$ is at least $\sigma$}\}.
$$
A merge at a time $t\ge 1$ uses a set of $\sigma$ $\Gppl^A$-edges that are incident 
to a single vertex and disjoint 
from the sets of $\Gppl^A$-edges used at other times. Moreover, the number of
clusters in the partition $\cP^1$ is at least $\ell-|\Epuzzle^A\cap \Epeople^A|$. 
It follows that for every $k\ge 0$ 
\begin{equation}\label{large-sigma-eq1}
\solve_A\subset \{|\Epuzzle^A\cap \Epeople^A|\ge \ell-k\}  \cup \left(\bigcup E_{v_1}\circ\ldots \circ E_{v_k}\right), 
\end{equation}
where the last union is over sets $\{v_1,\ldots, v_k\}\subset A$ of $k$ 
different vertices and the symbol $\circ$ represents disjoint occurrence of events.
Choose $k=\lceil \ell/2\rceil$. Then, as $|\Epuzzle^A|\le D\ell$, 
\begin{equation}\label{large-sigma-eq2}
\probsub{|\Epuzzle^A\cap \Epeople^A|\ge \ell-k}{p}\le \probsub{|\Epuzzle^A\cap \Epeople^A|\ge \ell/4}{p}
\le 2^{D\ell} p^{\ell/4}\le e^{-c(\log N)^2}, 
\end{equation}
for some constant $c>0$.
Moreover, by Lemma~\ref{binomial-easy-lem}, 
\begin{equation}\label{large-sigma-eq3}
P(E_v)\le \left(\frac{3\ell p}{\sigma}\right)^\sigma \le \left(\frac{6\lambda}{\sigma^2}\right)^\sigma
\end{equation}
and therefore by (\ref{large-sigma-eq1})--(\ref{large-sigma-eq3}) and BK inequality, 
$$
\probsub{\solve_A}{p}\le e^{-c(\log N)^2}+ \left(\frac{6\lambda}{\sigma^2}\right)^{\sigma k}
\le e^{-c(\log N)^2}+ \left(\frac{6\lambda}{\sigma^2}\right)^{\log N/2},
$$
provided that $\lambda\le \sigma^2/6$. Now,  
by Lemmas~\ref{is-double} and~\ref{clusters-lem}, for every $\sigma\ge 1$, 
\begin{equation}\label{large-sigma-eq4}
\probsub{\solve}{p}\le N\log N \left((eD)^{2\log N}e^{-c(\log N)^2}
+ \left(\frac{6e^4 D^4\lambda}{\sigma^2}\right)^{\log N/2}\right). 
\end{equation}
If 
$\lambda\le \sigma^2/(6 e^8 D^4)$, then (\ref{large-sigma-eq4}) implies that 
$\probsub{\solve}{p}\to 0$ as $N\to \infty$. 
\end{proof}

\section{Two-dimensional torus puzzle: AE dynamics}

For the rest of the paper, we assume that the puzzle graph is the two-dimensional lattive torus 
with $V=\bZ_n^2$. To make sure that $p_c$ is small, we also assume that $\tau$ is either 1 or 2. Indeed,
when $\tau\ge 3$, $\probsub{\solve}p$ is close to 
$0$ unless $p$ is close to 1. To see this, assume there is a $2\times 2$ square  $S$ none of whose
four vertices are doubly connected to any other vertex. Then each of the four vertices of $S$ is 
in its own singleton cluster in $\final$ and therefore $\solve$ cannot happen. Thus 
$\probsub{\solve}p$ is exponentially small if $p$ is bounded away from 1. 

In this section we assume $\tau=1$ and $\theta=\infty$ 
(except in Proposition~\ref{theta3-prop}
at the very end). It follows immediately 
from Theorem~\ref{intro-ring-thm} and Theorem~\ref{intro-lb-general} that the order parameter 
is $p\log n$ for all $\sigma$. To prove Theorem~\ref{intro-2d-bounds}, we will further restrict our attention to $\sigma=1$, 
and assume $p=\frac{\lambda}{\log n}$. . 
 
To get the lower bound, we recall that the number of connected subsets on $\bZ^2$ 
of size $k$ that contain the origin is at most $4.65^k$, for large $k$
(see Section 3 of \cite{Fin}). Moreover, by Theorem 3.2 in \cite{OCo}, for 
any $c,\epsilon>0$,
the \ER graph with $n$ vertices and edge probability $c/n$ is connected 
with probability at most $(1-e^{-c}+\epsilon)^n$ for a large enough $n$.

It follows that, for large $n$,
$$
\prob{\solve}\le c n^2\log n\sup_{\alpha\in [1,2]} (4.65)^{\alpha c \log n} 
\left(1-e^{-\lambda\alpha c}+\epsilon\right)^{\alpha c \log n}, 
$$
which proves the following result. 

\begin{lemma}\label{lb-2d}
Assume that $\lambda$ is any number such that, for some $c>0$, 
$$
\inf_{\alpha\in [1,2]}\left(-\alpha c \log(1-e^{-\lambda\alpha c})-\alpha c \log 4.65\right)>2.
$$
Then $\prob{\solve}\to 0$; in fact, a.~a.~s.~there is no internally solved set of 
size at least $c\log n$. 
\end{lemma}

It is easy to use Theorem 1 of \cite{BCDS} (or Theorem~\ref{intro-ring-thm}) to
show that $\limsup_n p_c\log n\le \pi^2/12$. We now establish a significant improvement 
of this upper bound by using a percolation comparison similar to the one in Section 3 
of \cite{FGG}.

Fix integers $k>0$ and $\ell\ge 0$. 
For a probability $r\in (0,1)$, assume that the origin $(0,0)$ 
is open, every other point in $Q_k=\{(x,y)\in \bZ_+^2: x+y\le k\}$ 
is open independently with probability $r$, and no point in $Q_k^c$
is open. Let $\phi_{k,\ell}(r)$ be the probability that there is
a connected cluster of at least $k+1+\ell$ open sites containing 
$(0,0)$ and the line $x+y=k$. Observe that $\phi_{k,0}(r)$ is 
the probability of a (possibly unoriented) connection between $(0,0)$ and 
the line $x+y=k$ within $Q_k$. Further, $\phi_{k,0}(r)$ is bounded away from 
0 independently of $k$ as soon as $p>\pcsite$ \cite{Rus}; here, $\pcsite$ 
is the critical probability for site percolation. Finally, observe that each
$\phi_{k,\ell}$ is a polynomial of degree $k(k+3)/2$, readily computable for small $k$, 
and non-decreasing on $[0,1]$.

\begin{lemma}\label{ub-2d}
Assume that 
$$
\lambda>\frac 12\int_0^{-\frac{\log(1-\pcsite)}{k+\ell}}-\log\phi_{k,\ell}(1-\exp((k+\ell)r))\, dr,
$$ 
for some $k$ 
and $\ell$. Then  $\prob{\solve}\to 1$.
\end{lemma}

\begin{proof}
Observe that the integrand is a positive decreasing function and that the integral is 
finite. We will drop the subscripts $k$ and $\ell$ which are fixed throughout the proof. 

Pick $r_0>-\frac{\log(1-\pcsite)}{k+\ell}$ 
and $\epsilon>0$ so that
$$
\lambda\ge \epsilon+\lambda_0
$$
where
$$
\lambda_0=\epsilon+ \frac 12\int_0^{r_0}-\log\phi(1-\exp((k+\ell)r)\, dr.
$$
Again, we will use the fact that $\Gppl$ with probability $\lambda/\log n$ 
stochastically dominates the union of two independent \ER graphs, with probabilities 
$\lambda_0/\log n$ and $\epsilon/\log n$. 

Write 
$M=\lceil\log n/\lambda_0\rceil$, 
$J=\lceil Mr_0\rceil$. Let $G_1$ be the event 
that there is a cluster of at least $J(k+\ell)$ sites inside $Q_{Jk}$
that connects $(0,0)$ to the line $x+y=Jk$. Then
\begin{equation}\label{ub-2d-eq1}
\begin{aligned}
\probsub{G_1}{\lambda_0/\log n}
&\ge \phi(1/M)\cdot \prod_{j=1}^J\phi(1-\exp((k+\ell)j/M)\\
&=
\exp\left(\log\phi(1/M)-M\cdot\sum_{j=1}^J\frac 1M (-\log\phi(1-\exp((k+\ell)j/M)))\right)
\\
&\ge
\exp\left(\log\phi(1/M)-M\cdot\int_{1/M}^{J/M}-\log\phi(1-\exp((k+\ell)r)\, dr\right)
\\
&\ge 
\exp\left(-M\cdot\int_{0}^{r_0}-\log\phi(1-\exp((k+\ell)r)\, dr-C\log M\right), 
\end{aligned}
\end{equation}
for some constant $C$. 
Let $G_2$ be the event that
$Q_{M^3}$ contains an internally solved cluster 
connecting $(0,0)$ to a point on the line $x+y=M^3$.
As
$$
1-(1-\lambda_0/\log n)^{(k+\ell)J}\ge 1-\exp(-(k+\ell)r_0)>\pcsite,
$$
the classic result of Russo \cite{Rus} implies that there exists an 
$\alpha=\alpha(r_0)>0$ so that 
\begin{equation}\label{ub-2d-eq2}
\probsub{G_2\mid G_1}{\lambda_0/\log n}\ge \alpha.
\end{equation}

The square $[0,n-1]^2$ contains at least $0.5n^2/M^6$ disjoint translations of $Q_{M^3}$ and each of 
them independently contains a translate of the event $G_2$. Therefore, 
by (\ref{ub-2d-eq1}) and (\ref{ub-2d-eq2}), for large $n$,
\begin{equation}\label{ub-2d-eq3}
\begin{aligned}
&\probsub{\text{there is an internally solved cluster of size $\ge M^3$}}{\lambda_0/\log n}
\\
&\ge 
1-\left(1-\probsub{G_2}{\lambda_0/\log n}\right)^{0.5n^2/M^6}
\\&
\ge 1-\exp\left(-e^{-2\log n+\epsilon \log n/\lambda_0}\cdot \frac {0.5n^2}{M^6}\right)\\
&=1-\exp(-n^{\epsilon/\lambda_0}/M^6)\to 1,
\end{aligned}
\end{equation}
as $n\to\infty$.  

The final step uses sprinkling: for any fixed set $S\subset [0,n-1]^2$ of size $M^3$, 
\begin{equation}\label{ub-2d-eq4}
\begin{aligned}
\probsub{\text{$S$ is unstoppable}}{\epsilon/\log n}
\to 1,
\end{aligned}
\end{equation}
by Lemma~\ref{unstoppable-size}. 
Now Lemmas~\ref{unstoppable} and~\ref{dividing-up}, together with 
(\ref{ub-2d-eq3}) and (\ref{ub-2d-eq4}) finish the proof. 
\end{proof}

\begin{theorem}
\label{2d-bounds}
For a large enough $n$, 
$$\frac{0.0388}{\log n}< p_c < \frac {0.303}{\log n}.
$$
\end{theorem}

\begin{proof}[Proof of Theorem~\ref{intro-2d-bounds}]
The lower bound is obtained by talking $c=1.5116$ in Lemma~\ref{lb-2d},
which yields the infimum $2.008$ for $\lambda=0.0388$.
The upper bound is obtained by using
Lemma~\ref{ub-2d} with $k=6$, $\ell=4$ and 
the best rigorous upper bound for $\pcsite$ known,  $\pcsite<0.6795$ \cite{Wie}.
\end{proof}

We remark that one could also get a valid upper bound by allowing $\ell$ 
to change with $r$ in Lemma~\ref{ub-2d}. The resulting improvement 
in our constant is too small to justify additional complications. 


We end this section with a simple proposition that shows that only the $\theta=2$ 
case may have a different scaling when $\sigma=\tau=1$. 
Indeed, we show in Section 8 that it does. 

\begin{prop}\label{theta3-prop}
Assume that $\theta\ge 3$, while $\tau=\sigma=1$. Then 
$\probsub{\solve}{p}\to 0$ if $p<\frac 14\cdot~0.0388/\log n$.
\end{prop}

\begin{proof}
Assume 
for simplicity that $n$ is even and divide
the torus into $2\times 2$ squares. Create a new torus $\Gpuz'$ with a 
$(n/2)\times (n/2)$ vertex set $V'$. The new people graph $\Gppl'$
has an edge between $(i,j)$ and $(i',j')$ if and only if 
there is at least one $\Gppl$ edge connecting $2\times 2$ squares 
$(2i,2j)+\{0,1\}^2$ and $(2i',2j')+\{0,1\}^2$. Let $\solve'$ be the 
event that the AE jigsaw percolation solves the puzzle with the 
pair $\Gpuz'$, $\Gppl'$. As no point in a $2\times 2$ square on (the original) 
$\Gpuz$ has 3 neighbors outside it, it is easy to see that $\solve\subset\solve'$.
The result then follows from Theorem~\ref{2d-bounds}.  
\end{proof}

\section{Two-dimensional torus puzzle: $\theta=2$ and $\tau=1$}

In this section we determine how the scaling of $p_c$ depends on $\sigma$ 
when $\tau=1$ and $\theta=2$ and the puzzle graph is two-dimensional torus, 
proving Theorem~\ref{intro-theta2-thm}. We begin with the key step for 
the upper bound on $p_c$. 
  
\begin{lemma}\label{bp-local}
$$
\liminf_{p\to 0} p^{\sigma/(2\sigma+1)} \log\probsub{\grow}{p}\ge -2\nu_\sigma.
$$
\end{lemma}

\begin{proof} Let $B_k^h=[0,k]\times \{k+1\}$, 
$B_k^v=\{k+1\}\times [0,k]$. 
One scenario that assures that $\grow$ happens is that the pairs
$(k-1,0)$--$(k,0)$  and $(0,k-1)$--$(0,k)$ are doubly connected 
for $k=1,\ldots, \sigma+1$, and then 
for every $k>\sigma+1$ there are points $z_k\in B_k^h$ and $z_k'\in B_k^v$ with 
$\cpeople(z_k, [0,k]^2)\ge \sigma$ and $\cpeople(z_k', [0,k]^2)\ge \sigma$. Thus 
\begin{equation}\label{bp-local-eq1}
\begin{aligned}
\probsub{\grow}{p}
&\ge p^{2(\sigma+1)}\prod_{k=\sigma+1}^\infty \left[1-\prob{\Bin(k^2,p)<\sigma}^k\right]^2\\
&\ge p^{2(\sigma+1)}\prod_{k=\sigma+1}^\infty 
\left[1-\exp(-k\cdot\prob{\Bin(k^2,p)\ge\sigma})\right]^2.
\end{aligned}
\end{equation}
Fix an $\epsilon>0$. When $k\ge p^{-1/2+\epsilon}$ and $p$ is small enough, 
\begin{equation*}
\begin{aligned}
\prob{\Bin(k^2,p)\ge\sigma}&\ge \prob{\Bin(p^{-1+2\epsilon},p)\ge\sigma}\\
&\ge \prob{\Bin(p^{-1+2\epsilon},p)=\sigma}\ge cp^{2\epsilon \sigma}, 
\end{aligned}
\end{equation*}
for some constant $c>0$ that depends on $\sigma$. Therefore 
\begin{equation}\label{bp-local-eq2}
\begin{aligned}
&\prod_{k\ge p^{-1/2+\epsilon}} \left[1-\exp(-k\cdot\prob{\Bin(k^2,p)\ge\sigma})\right]\\
&\ge \exp\left(\sum_{k\ge 1}\log\left[ 1-\exp(-ckp^{2\epsilon\sigma})\right]  \right)\\
&\ge \exp\left(-p^{-2\epsilon\sigma}\int_0^\infty g(cx)\,dx\right)\\
&\ge \exp\left(-c^{-1}p^{-2\epsilon\sigma}\right).
\end{aligned}
\end{equation}
Moreover, when $\sigma+1\le k\le p^{-1/2+\epsilon}$ and $p$ is small enough,
\begin{equation*}
\begin{aligned}
k\cdot \prob{\Bin(k^2,p)=\sigma}&\ge \frac{(k-\sigma)^{2\sigma+1}}{\sigma!}
p^{\sigma} (1-p)^{p^{-1+2\epsilon}}\\
&\ge (1-\epsilon)\frac{(k-\sigma)^{2\sigma+1}}{\sigma!}
p^{\sigma},  
\end{aligned}
\end{equation*}
and so 
\begin{equation}\label{bp-local-eq3}
\begin{aligned}
&\prod_{k=\sigma+1}^{p^{-1/2+\epsilon}} \left[1-\exp(-k\cdot\prob{\Bin(k^2,p)\ge\sigma})\right]\\
&\ge \exp\left(\sum_{k\ge 1}\log
\left[ 1-\exp\left(-(1-\epsilon)\frac{k^{2\sigma+1}}{\sigma!}p^\sigma\right)\right]  \right)\\
&\ge \exp\left(-p^{-\sigma/(2\sigma+1)}\int_0^\infty 
g\left((1-\epsilon)\frac{x^{2\sigma+1}}{\sigma!}\right)\,dx\right).
\end{aligned}
\end{equation}
If $\epsilon<1/(4\sigma+2)$, (\ref{bp-local-eq1})--(\ref{bp-local-eq3}) imply that 
$$
\liminf_{p\to 0} p^{\sigma/(2\sigma+1)} \log \probsub{\grow}{p}\ge
-2\int_0^\infty 
g\left((1-\epsilon)\frac{x^{2\sigma+1}}{\sigma!}\right)\,dx.
$$
Now we send $\epsilon\to 0$ to get the desired inequality.
\end{proof}

Observe that if a subset of the vertex set $V$ is internally solved, so is the smallest rectangle 
containing it, therefore we will exclusively deal with internally solved rectangles 
in the rest of this section. 
The maximum (resp.~minumum) 
of two dimensions of a rectangle $R$ will be denoted by $\lng(R)$ (resp.~$\srt(R)$). As we will see, 
after the 
next five lemmas are established, the 
proof of the lower bound on $p_c$ proceeds by a variant of the argument 
in \cite{Hol}.

\begin{lemma} \label{theta2-supercr}
Assume $\lambda>\nu_\sigma^{2+1/\sigma}$. If $p\ge \lambda/(\log n)^{2+1/\sigma}$, then 
$\probsub{\solve}{p}\to 1$. 
\end{lemma}

\begin{proof} By Lemma~\ref{unstoppable-size}, for any fixed set $S$ of size $(\log n)^5$, and 
any $\epsilon>0$, 
\begin{equation}\label{theta2-supercr-eq1}
\probsub{\text{$S$ is unstoppable}}{\epsilon/(\log n)^{2+1/\sigma}}\to 1.
\end{equation}
Assume $\lambda^{\sigma/(2\sigma+1)}>\lambda''>\nu_\sigma$. 
Divide the torus into disjoint $(\log n)^5\times (\log n)^5$ squares. Call 
such a square {\it good\/}
if  the local jigsaw process, started from its lower left corner, 
produces an internally solved rectangle whose
longest side has length $(\log n)^5$. By Lemma~\ref{bp-local}, each of these 
squares is good with probability at least $\exp(-2\lambda''p^{-\sigma/(2\sigma+1)})$, 
independently of others. 
Then 
\begin{equation}\label{theta2-supercr-eq2}
\probsub{\text{there is a good square}}{p}
\ge 1-(1-n^{-2\lambda''\lambda^{-\sigma/(2\sigma+1)}})^{n^2/(\log n)^{10}}\to 1.
\end{equation}
The two inequalities (\ref{theta2-supercr-eq1}) and (\ref{theta2-supercr-eq2}), 
together with Lemma~\ref{unstoppable}, finish the proof. 
\end{proof}

\begin{lemma}\label{double-theta2}
If $\solve$ happens, there exists an internally solved rectangle $R$ with 
$\lng(R)\in[\log n, 2\log n]$.
\end{lemma}

\begin{proof} The argument is the same as for Lemma~\ref{is-double}. 
\end{proof}

For a rectangle $R$, we say that its column $I$ is \df{isolated}
if, for every point $v\in I$, $\cpeople(v, R\setminus I)<\sigma$ and $v$ is not doubly connected 
to any point in $R\setminus I$. Further, 
we say that $I$ is \df{inert} if no two vertices within $I$ are doubly connected. 

\begin{lemma}\label{column-condition} Assume $R$ is a rectangle with at least 
two columns. If $\sigma>1$, 
$\solve_R\subset\{$no column of $R$ is both isolated and inert$\}$, while if 
$\sigma=1$, 
$\solve_R\subset\{$no column of $R$ is isolated$\}$.
\end{lemma}

\begin{proof}
If $\sigma>1$, let  
$I$ be a column that is both isolated and inert and  
assume that, at some time $t$, all points in $I$ are in singleton clusters. 
If $\sigma=1$, assume that $I$ is isolated and assume that at time $t$ no cluster
intersects both $I$ and $R\setminus I$. In either case 
it is easy to see that the condition remains true at time $t+1$. 
\end{proof}
 
\renewcommand{\index}{{\tt index}}
\newcommand{\tg}{{\tt tg}}
\newcommand{\ba}{{\mathbf a}}
\newcommand{\bb}{{\mathbf b}}
\newcommand{\bE}{{\mathbf E}}
\newcommand{\bA}{{\mathbf A}}
\newcommand{\pis}{p_{\mathrm{nis}}}
\newcommand{\pin}{p_{\mathrm{nin}}}
\newcommand{\rset}{\tt r\_set}
\newcommand{\pmax}{p_{\mathrm{max}}}
\newcommand{\psucc}{p_{\mathrm{succ}}}

\begin{lemma}\label{seed-theta2} Fix $b>0$ and $Z\in (0,b)$. There exists a 
constant $C$ dependent only on $\sigma$ so that for small enough $p$ 
the following is true. 
For any rectangle $R$ with $\lng(R)\le bp^{-\sigma/(2\sigma+1)}$ and $\srt(R)\le Zp^{-\sigma/(2\sigma+1)}$, 
$$
\probsub{\text{$R$ is internally solved}}{p}\le (C b^{\sigma}Z^{\sigma+1})^{\lng(R)/(\sigma+1)}. 
$$
\end{lemma}

 \begin{proof}
For a vertex $v\in R$, define the event
$$
E_v=\{v\text{ is doubly connected to another vertex in }R\}\cup
\{\cpeople(v, R\setminus\{v\})\ge \sigma\}.
$$
Then, by Lemma~\ref{binomial-easy-lem},
\begin{equation}\label{seed-theta2-eq1}
\probsub{E_v}{p}\le 4p+\left(\frac 3\sigma\right)^\sigma (|R|p)^\sigma.
\end{equation}
Assuming that $\lng(R)$ is the number of columns of $R$, by
Lemma~\ref{column-condition},
\begin{equation}\label{seed-theta2-eq2}
\probsub{\text{$R$ is internally solved}}{p}\le \probsub{\sum_v
\indicator_{E_v}\ge \lng(R)}{p}.
\end{equation}
Let $k=\lceil \lng(R)/(\sigma+1)\rceil$. As each $E_v$ only requires the
presence of at most $\sigma$ $\Gppl$-edges connecting
at most $\sigma+1$ vertices,
$$
\{\sum_v \indicator_{E_v}\ge \lng(R)\}\subset \bigcup
\left( E_{v_1}\circ \ldots \circ E_{v_k}\right),
$$
where the union is over all subsets $\{v_1,\ldots, v_k\}\subset R$
of size $k$, and the symbol $\circ$ denotes disjoint occurrence of
the events. Therefore, by the
BK inequality,
(\ref{seed-theta2-eq1}) and
Lemma~\ref{binomial-easy-lem}
\begin{equation}\label{seed-theta2-eq3}
\begin{aligned}
&\probsub{\sum_v \indicator_{E_v}\ge \lng(R)}{p}
\\&\le \binom{|R|}{k}\left( 4p+\left(\frac 3\sigma\right)^\sigma
(|R|p)^\sigma\right)^k
\\&\le \left[12(\sigma+1)p\,\srt(R)+
3(\sigma+1)\left(\frac
3\sigma\right)^\sigma\srt(R)^{\sigma+1}\lng(R)^\sigma
p^\sigma\right]^{\lng(R)/(\sigma+1)}.
\end{aligned}
\end{equation}
Now, $p\,\srt(R)=o(1)$ and $\srt(R)^{\sigma+1}\lng(R)^\sigma
p^\sigma\le Z^{\sigma+1}b^\sigma$, so (\ref{seed-theta2-eq3}) and
(\ref{seed-theta2-eq2})
finish the proof.
\end{proof}

We pause in our quest to prove Theorem~\ref{intro-theta2-thm} to see how our results up to this point 
imply a weaker result:
$p_c$ is between two constants times $1/(\log n)^{2+1/\sigma}$.
Indeed, we can use Lemmas~\ref{double-theta2} and~\ref{seed-theta2}
to get, for any $b$ and $\lambda$, with $p=\lambda/(\log n)^{2+1/\sigma}$,  
$$
\probsub{\solve}{p}\le Cn^2(\log n)(Cb^{2\sigma+1})^{\frac b{2(\sigma+1)}\lambda^{-\sigma/(2\sigma+1)}\log n}.
$$
Now, we first choose $b$ small enough so that $Cb^{2\sigma+1}<e^{-1}$,
and then $\lambda$ so small that
$$\frac b{2(\sigma+1)}\lambda^{-\sigma/(2\sigma+1)}>2$$
to ensure that 
$\probsub{\solve}{p}\to 0$.

\newcommand{\pp}{\sigma/(2\sigma+1)}

The key to proving the sharp transition is the next lemma, which 
gives an adequately sharp upper bound on the probability that the solving 
progresses from a rectangle with sides on the scale 
$p^{-\pp}$ to a rectangle slightly larger on the same scale.   
 
\begin{lemma}\label{D-theta2} Fix small $a,\epsilon>0$ and large $b>0$. Then there 
exists a $\delta>0$ so that the following holds uniformly over $x,y\in [a,b]$. 
Assume that $R\subset R'$ are rectangles with dimensions $xp^{-\pp}\times yp^{-\pp}$
and $(x+\delta_x)p^{-\pp}\times (y+\delta_y)p^{-\pp}$, with $\delta_x,\delta_y<\delta$. 
Then, for a small enough $p$,
$$
p^{\pp}\log\probsub{D(R,R')}{p}\le-(1-\epsilon)
(g\left(\textstyle{\frac 1{\sigma!}}x^\sigma y^{\sigma+1}\right)\delta_x+g\left(
\textstyle{\frac 1{\sigma!}}x^{\sigma+1}y^\sigma\right)\delta_y)
$$ 
\end{lemma}

\begin{proof}
Divide $R'\setminus R$ into eight disjoint rectangles $S_1,\ldots, S_8$ 
as in Figure~\ref{r1r8}.

\begin{figure}
\begin{center}
\includegraphics[width=.4\textwidth]{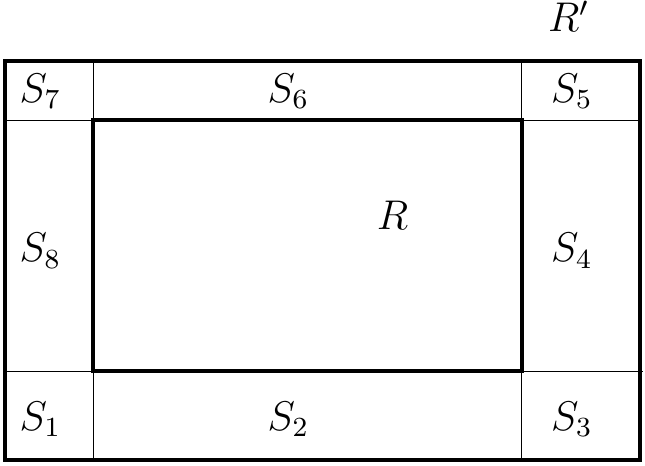}
\end{center}
\caption{The rectangles $S_1,\ldots, S_8$ in the proof of Lemma~\ref{D-theta2}.}
\label{r1r8}
\end{figure}

\renewcommand{\pp}{\sigma/(2\sigma+1)}

Let $S_h=S_1\cup S_2\cup S_3\cup S_5\cup S_6 \cup S_7$, $S_v=S_7\cup S_8\cup S_1\cup S_3\cup S_4 \cup S_5$, and $S_c=S_1\cup S_3 \cup S_5\cup S_7$. 
Call a vertex $v\in R'\setminus R$ (resp.~$v\in S_v$) {\it exceptional\/}
(resp.~{\it horizontally exceptional\/}) if it is either 
doubly connected to another vertex in $R'$, or it has both 
$\cpeople(v, R')\ge \sigma$ 
and $\cpeople(v, S_h\cup S_v)\ge 1$ (resp.~$\cpeople(v, S_h)\ge 1$). Moreover, 
declare $v$ {\it successful\/} (resp.~{\it horizontally successful\/}) if $\cpeople(v, R)\ge \sigma$
(resp.~$\cpeople(v, R'\setminus S_h)\ge \sigma$). 

Without loss of generality, we may assume $\delta_x\le \delta_y$. We divide our argument 
into two cases. 

\noindent{\it Case 1\/}: $\delta_x\ge \epsilon \delta_y$. 

Define the following events 
\begin{equation}\label{D-theta2-events}
\begin{aligned}
&G_1=\{\text{at least $\epsilon\delta_y p^{-\pp}$ vertices in $S_h$ are exceptional}\},\\
&G_2=\{\text{at least $\epsilon\delta_xp^{-\pp}$ vertices in $S_v$ are exceptional}\},\\
&G_3=\{\text{at least $\epsilon(\delta_x\wedge \delta_y) p^{-\pp}$ vertices in $S_c$ 
are successful}.\}\\
&G_4=\{\text{at least $(1-3\epsilon)\delta_y p^{-\pp}$ rows in $S_2\cup S_6$ 
contain a successful vertex}.\}\\
&G_5=\{\text{at least $(1-3\epsilon)\delta_x p^{-\pp}$ columns in $S_4\cup S_8$ 
contain a successful vertex}.\}
\end{aligned}
\end{equation}
By Lemma~\ref{column-condition}, 
\begin{equation}\label{D-theta2-inclusion}
\begin{aligned}
D(R,R')\subset G_1\cup G_2 \cup G_3 \cup (G_4\cap G_5).
\end{aligned}
\end{equation}
From now on, $C$ will be a generic constant that depends on $a$, $b$, and $\sigma$.
We have, for any vertex~$v$,
$$
\probsub{v\text{ is exceptional}}{p}\le C\delta p^{\pp}.
$$
As in the proof of Lemma~\ref{seed-theta2}, on $G_1$ there must exist 
${\frac 1{\sigma+1}\epsilon\delta_y p^{-\pp}}$ vertices that are exceptional 
disjointly, and analogous statement holds for $G_2$. Therefore, by Lemma~\ref{binomial-easy-lem}, 
\begin{equation}\label{D-theta2-eq1}
\begin{aligned}
&\probsub{G_1}{p}\le (C\delta/\epsilon)^{\frac 1{\sigma+1}\epsilon\delta_y p^{-\pp}},\\
&\probsub{G_2}{p}\le (C\delta/\epsilon)^{\frac 1{\sigma+1}\epsilon\delta_x p^{-\pp}}. 
\end{aligned}
\end{equation}
Moreover, for $p$ small enough, 
\begin{equation}\label{exceptional}
\psucc=\probsub{v\text{ successful}}{p}\le (1+\epsilon) \frac 1{\sigma!} (xy)^\sigma p^{\pp}
\le C p^{\pp},
\end{equation}
so that, by Lemma~\ref{binomial-easy-lem}, as the points in $S_c$ are successful independently,
\begin{equation}\label{D-theta2-eq2}
\begin{aligned}
&\probsub{G_3}{p}\le (C\delta/\epsilon)^{\epsilon\delta_x p^{-\pp}}
\end{aligned}
\end{equation}
Now, as points in $S_2\cup S_4\cup S_6\cup S_8$ are also successful independently, 
$G_4$ and $G_5$ are independent. To estimate $\probsub{G_4}{p}$, we see 
that the number of choices of the required number of rows that contain a 
successful vertex is bounded above by 
$\exp(C\epsilon\log \frac 3\epsilon\delta_yp^{-\pp})$, which we will, for
simplicity, bound by  $\exp(C\sqrt{\epsilon}\delta_yp^{-\pp})$.
Moreover, 
for $p$ small enough, by (\ref{exceptional}),
\begin{equation}\label{D-theta2-eq3}
\begin{aligned}
\probsub{G_4}{p}&\le
\exp(C\sqrt{\epsilon}\delta_yp^{-\pp})\left(1-(1-\psucc)^{xp^{-\pp}}\right)^{(1-3\epsilon)\delta_yp^{-\pp}}
\\
&\le \exp\left[(C\sqrt{\epsilon}\delta_yp^{-\pp}
 -g\left((1+\epsilon)\psucc xp^{-\pp}\right)(1-3\epsilon)\delta_yp^{-\pp}\right]\\
&\le 
\exp\left[(C\sqrt{\epsilon}\delta_yp^{-\pp} -g\left((1+\epsilon)^2 \textstyle{\frac 1{\sigma!}} x^{\sigma+1}y^\sigma\right)(1-3\epsilon)\delta_yp^{-\pp} \right]\\
&\le 
\exp\left[ -(1-C\sqrt\epsilon) 
g\left(\textstyle{\frac 1{\sigma!}}x^{\sigma+1}y^\sigma\right)\delta_yp^{-\pp}
 \right].
\end{aligned}
\end{equation}
and 
\begin{equation}\label{D-theta2-eq4}
\begin{aligned}
\probsub{G_5}{p} \le 
\exp\left[ -(1-C\sqrt\epsilon) 
g\left(\textstyle{\frac 1{\sigma!}}x^{\sigma}y^{\sigma+1}\right)\delta_xp^{-\pp}
 \right].
\end{aligned}
\end{equation}
Let $\beta$ be the upper bound on $\probsub{G_4\cap G_5}{p}=\probsub{G_4}{p}\probsub{G_5}{p}$
obtained by (\ref{D-theta2-eq3}) and (\ref{D-theta2-eq4}). 
Now we claim that 
$\probsub{G_1}{p}$, $\probsub{G_2}{p}$, and 
$\probsub{G_3}{p}$ are, for small enough 
$\delta$, all smaller than $\beta$. It is here that we use 
the {\it Case 1\/} assumption. For example, for arbitrary large $M>0$, 
$\delta$ can be chosen small enough so that 
\begin{equation*}
\begin{aligned}
&\probsub{G_2}{p}\le (C\delta/\epsilon)^{\frac 1{\sigma+1}\epsilon^2\delta_y p^{-\pp}}
\le \exp(-M\delta_yp^{-\pp}), 
\end{aligned}
\end{equation*}
while 
$$
\beta\ge \exp(-C\delta_yp^{-\pp}).
$$
Therefore, for small enough $\delta$, 
$\probsub{D(R,R')}{p}\le 4\beta$, 
which finishes the proof in this case.

\noindent{\it Case 2\/}: $\delta_x\le \epsilon\delta_y$. 

In this case it is enough to show 
\begin{equation}\label{D-theta2-eq5}
\begin{aligned}
&\probsub{D(R,R')}{p}\le\exp\left[ -(1-C\sqrt\epsilon) 
g\left(\textstyle{\frac 1{\sigma!}}x^{\sigma+1}y^\sigma\right)\delta_yp^{-\pp}
 \right]
\end{aligned}
\end{equation}
as, for $x,y\in [a,b]$, 
\begin{equation*}
\begin{aligned}
g\left(\textstyle{\frac 1{\sigma!}}x^{\sigma}y^{\sigma+1}\right)\delta_x
\le C\epsilon g\left(\textstyle{\frac 1{\sigma!}}x^{\sigma+1}y^{\sigma}\right)\delta_y.
 \end{aligned}
\end{equation*}
To demonstrate (\ref{D-theta2-eq5}), introduce the following two events 
\begin{equation*}
\begin{aligned}
&G_6=\{\text{at least $\epsilon\delta_y p^{-\pp}$ vertices in $S_h$ are horizontally exceptional}\},\\
&G_7=\{\text{at least $(1-\epsilon)\delta_y p^{-\pp}$ rows in $S_h$ 
contain a horizontally successful vertex}\}.\\
\end{aligned}
\end{equation*}
Now $P(D(R,R'))\le P(G_6)+P(G_7)$ and the rest of the proof is similar as in {\it Case 1\/}.
\end{proof}

\begin{proof}[Proof of Theorem~\ref{intro-theta2-thm}]
The upper bound on $p_c$ follows from Lemma~\ref{theta2-supercr}. 
The proof of the lower bound, at this point,
follows rather closely the argument in Sections 6--10 in \cite{Hol}
and we merely identify the key steps. 
The functional $w$ on paths $\gamma$ is now given by 
$$
w(\gamma)=\int_\gamma \left(g\left(\frac{x^{\sigma} y^{\sigma+1}}{\sigma!}\right)\,dx+g\left(\frac{x^{\sigma+1}y^{\sigma}}{\sigma!}\right)\,dy\right),
$$ 
and analogous variational principles as in Section 6 of \cite{Hol} hold. 
The disjoint spanning properties and hierarchies also have analogous 
formulations, and then the argument in Section 10 of \cite{Hol} goes 
through by the use of key Lemmas~\ref{seed-theta2} and \ref{D-theta2}. 
\end{proof}
 

\section{Two-dimensional torus puzzle: $\tau=2$}

Here we assume the two-dimensional torus with $\tau=2$. We will assume 
that
$\theta\ge 2$ and $\sigma\ge 1$ are arbitrary and
show that the asymptotic scaling of the critical probability is always 
$1/\log n$, proving Theorem~\ref{intro-tau2-theorem}.
We begin with the local result.

\begin{lemma}\label{tau2-local} Assume that $\tau=2$, $\theta=\infty$ 
and $\sigma\ge 1$.
Then
$$
\liminf_{p\to 0} p\log 
\probsub{\grow}{p}\ge-\frac{\pi^2}3+\int_0^\infty\log 
\prob{\Poisson(x)\ge \sigma}\, dx.
$$
\end{lemma}

\begin{proof} Fix $\epsilon>0$. For a $b>0$ (which will depend on 
$\e$), let $J=\lceil (b/p)^{1/2}\rceil$.
Let $G_1$ be the event that the pairs of points $\{(0,j-1), (0,j)\}$
and $\{(j-1, 0), (j,0)\}$, $1\le j\le J$ are all
doubly connected.

Order the points in $\bZ_+^2$ as in the proof of Corollary~\ref{Cor:Z2 long range}: $(x_1, y_1) < (x_2, y_2)$ if either $x_1 + y_1 < x_2 + 
y_2$ or
$x_1 + y_1 = x_2 + y_2$ and
$x_1< x_2$. Let $G_2$ be the event that every point $(x,y)\in 
(0,J]^2$ has at
least $\sigma$ $\Gppl$-neighbors within $([0,J]\times\{0\})\cup 
(\{0\}\times [0,J])\cup \overleftarrow{(x,y)}$. Here, 
$\overleftarrow{(x,y)}$ is the set of points that 
{\it strictly\/} precede $(x,y)$
in the ordering.

As in the proof of Lemma~\ref{bp-local}, let $B_k^h=[0,k-1]\times 
\{k\}$ and
$B_k^v=\{k\}\times [0,k-1]$. Let $G_3$ be the event that, for every $k> 
J$,
there are points $z_k\in B_k^h$ and $z_k'\in B_k^v$, each of which
is doubly connected to a point in $[0,k-1]^2$,  and let $G_4$ be the 
event that,
for every $k>J$, each point in
$B_k^h\cup B_k^v\cup\{(k,k)\}$ is $\Gppl$-connected to at least 
$\sigma$ points in
  $[0,k-1]^2$.

By the FKG inequality, $P(G_3\cap G_4)\ge P(G_3)P(G_4)$, while $G_1$, 
$G_2$, and
$G_3\cap G_4$ are independent. It is easy to see that
$G_1\cap G_2 \cap G_3\cap G_4\subset \grow$. Therefore,
\begin{equation}\label{tau2-local-eq1}
\probsub{\grow}{p} \ge 
\probsub{G_1}{p}\probsub{G_2}{p}\probsub{G_3}{p}\probsub{G_4}{p},
\end{equation}
and we estimate each factor separately.
Clearly
\begin{equation}\label{tau2-local-eq2}
\probsub{G_1}{p}=p^{2J},
\end{equation}
and, by the same estimates as in (\ref{G2-est}) and (\ref{G3-est}),
\begin{equation}\label{tau2-local-eq3}
\probsub{G_2}{p}\ge \exp\left(-2\epsilon p^{-1} + p^{-1}\int_0^\infty\log 
\prob{\Poisson(x)\ge \sigma}\, dx\right),
\end{equation}
for small enough $p$.
Further,
\begin{equation}\label{tau2-local-eq4}
\begin{aligned}
\probsub{G_3}{p}&=\prod_{k=J+1}^\infty(1-(1-p)^k)^2
&\ge \prod_{k=1}^\infty(1-(1-p)^k)^2
&\ge \exp\left(-p^{-1}\cdot\frac{\pi^2}{3}\right),
\end{aligned}
\end{equation}
by the standard calculation, and
\begin{equation}\label{tau2-local-eq5}
\begin{aligned}
\probsub{G_4}{p}
&=\prod_{k=J+1}^\infty \prob{\Bin((k-1)^2, p)\ge \sigma}^{2k+1}\\
&\ge\prod_{k=J+1}^\infty
\prob{\Bin(\lfloor(k-1)^2/\sigma\rfloor, p)\ge 1}^{(2k+1)\sigma}\\
&\ge
\exp\left(3\sigma\sum_{k=J+1}^\infty 
k\log(1-e^{-0.5\sigma^{-1}pk^2})\right)\\
&=
\exp\left(3\sigma {p}^{-1}
\sum_{k=J+1}^\infty k\sqrt p\cdot \log(1-e^{-0.5\sigma^{-1}(k\sqrt 
p)^2})\sqrt p\right)\\
&\ge
\exp\left(3\sigma { p}^{-1}\int_{b}^\infty 
x\log(1-e^{-0.5\sigma^{-1}x^2})\, dx\right)\\
&\ge
\exp\left(-\epsilon p^{-1} \right),
\end{aligned}
\end{equation}
for small enough $p$ and large enough enough $b$.
The result now follows from 
(\ref{tau2-local-eq1})--(\ref{tau2-local-eq5}).
\end{proof}

\begin{proof}[Proof of Theorem~\ref{intro-tau2-theorem}]
We first consider the parameter choice $\theta=2$, which makes the 
growth easiest.
The resulting analysis is
also the easiest, as the dynamics is a slight
variant of the modified bootstrap percolation \cite{Hol}. Namely, it is equivalent to the 
following edge-growth process. Initially,  the edges
that connect doubly connected vertices are occupied. Then one 
simply ``completes the squares,'' i.e., when two $\Gpuz$-edges $\{v,v_1\}$,
$\{v,v_2\}$ adjacent to the vertex $v$ are occupied, there exists a
unique $v'$ so that the two
$\Gpuz$-edges $\{v',v_1\}$, $\{v',v_2\}$ are adjacent and each (if not
already occupied) becomes occupied at the next time.  
By following the argument from \cite{Hol}, (\ref{intro-tau2-theta2}) follows.

Therefore, for any $\epsilon>0$, and
$p\le (1-\epsilon)(\pi^2/6)/\log n$,
$\probsub{\solve}{p}\to 0$
in all cases, proving the lower bound in (\ref{intro-tau2-general}).
The upper bound in (\ref{intro-tau2-general}) follows immediately from 
Lemma~\ref{tau2-local}.
\end{proof}

 
\section{Two-dimensional torus puzzle: computational aspects}

For concreteness, we assume the AE dynamics throughout this section.
This is more challenging to simulate than the basic jigsaw percolation
\cite{BCDS} as a cluster cannot be ``collapsed'' into a vertex.
For large two-dimensional tori $\bZ_n^2$, the simulations seem daunting at first, as
generation of $\Gppl$ alone involves $n^2(n^2-1)/2$ coin flips. However,
as we will see, only a small proportion of these flips is ever likely to
be used, leading to a significant reduction in computational requirements.
The key idea is that the status of edges of $\Gppl$ can be
determined dynamically as needed, rather than at the beginning.

We begin by describing an implementation of the dynamics.
We recall that the state at time $t$
is a partition $\cP^t=\{W_i^t: i=1,\dots, I_t\}$ of $\bZ_n^2$ into
disjoint nonempty sets that are internally solved; in particular, they are
connected clusters in both graphs. One may use an appropriate
pointer-based data structure which makes {\tt Union} and {\tt Find} operations efficient,
for example shallow (threaded) trees \cite{Sed}. We change the terminology
slightly in that we consider $\Gppl$ a random configuration on the set of complete
graph edges in which each edge is independently {\it open\/} with
probability $p$ and {\it closed\/} otherwise. At any time $t=0, 1,\ldots$ one
performs the following two operations:
\begin{itemize}
\item[(1)] For any point $z\in W_i^t$ that has a $\Gpuz$-neighbor in a cluster
$W_j^t$,  $j\ne i$, check the status of the $\Gppl$--edges between $z$ and
points $z'\in W_j^t$ in some order; stop when an open edge is encountered or when 
all points in $W_j^t$ are exhausted.
In the former case say that $z$ {\it communicates\/} with $z'$.
\item[(2)] Repeatedly merge any two sets in the partition if a point in one
communicates with a point in another, until no more merges are possible. The resulting
partition is $\cP^{t+1}$.
\end{itemize}

Whenever a status of $\Gppl$-edge in step (1) is checked,
we say that an {\it oriented\/} pair $z\to z'$ is \df{examined} at time $t$. Observe that
$z\to z'$ and $z'\to z$ may be examined at the same time $t$. Observe also that
if a pair $z\to z'$ is examined at time $t-1$, and $z\to z'$ or $z'\to z$ is again examined
at time $t$, then $\{z,z'\}$ is necessarily a closed edge
in $\Gppl$. One may arrange the algorithm so that no edge is examined twice. 
However, in a practical implementation, it is easiest to store 
the set of edges $(z,z')$ of $\Gppl$, such that
either $z\to z'$ or $z'\to z$, in a convenient data structure
(say, a binary search tree or a hash table \cite{Sed}). These are the edges whose status has been
decided.

\begin{theorem} Fix any sequence of probabilities $p$.
With probability converging to 1 as $n\to\infty$, for every vertex $z\in V$,
at most $1000(\log n)^2$ oriented pairs $z\to z'$ are ever examined.
Consequently, the space and time
requirements for deciding whether the puzzle is solved are
both a.~a.~s.~bounded by $Cn^2(\log n)^2$, for some absolute constant $C$.
\end{theorem}

\begin{proof} We may assume that $p\ge 0.038/\log n$, as otherwise the
result follows from Lemma~\ref{lb-2d} and Theorem~\ref{2d-bounds} (in fact,
with $\log n$ instead of $(\log n)^2$).

Fix an edge $\{z_1,z_2\}\in \Epuzzle$. Call this edge
{\it active\/} at time $t$ if $z_1$ and $z_2$ belong to different clusters at time $t$.
Observe that no pair $z\to z'$ is checked at any time $t'\ge t$ if none of the
for $\Gpuz$-edges incident to $z$ are active at time $t$. Furthermore, if $W_1^t, W_2^t\in \cP^t$
are the clusters that contain $z_1,z_2$, respectively, then for $t\ge 0$,
$$
\{\{z_1, z_2\}\text{ active at time $t+1$}\}\subset\{\{z_1, z_2'\}\notin \Epeople,
\{z_2, z_1'\}\notin \Epeople,\text{ for all }z_1'\in W_1^t, z_2'\in W_2^t\},
$$
as the status of $\Gppl$-edges $\{z_1, z_2'\}$ is checked when $z_2'$ joins the cluster
containing $z_2$. It follows that
$$
\probsub{\{z_1, z_2\}\text{ is active at time $t+1$}, |W_1^t\cup W_2^t|\ge k}{p}
\le (1-p)^k,
$$
and then
$$
\probsub{\{z_1,z_2\}\text{ is active at time $t+1$}, |W_2^t|\ge 240(\log n)^2}{p}
=o(n^{-4}).
$$
Let
$$
G_{z_1, z_2}=\{\text{more than
$240(\log n)^2$ pairs $z_1\to z_2'$, $z_2'\in \bigcup_{s\ge 0}W_2^s$ are examined}\}.
$$
Then, for any fixed time $t$,
$$
\probsub{G_{z_1, z_2}\cap\{|W_2^t|\ge 240(\log n)^2\}}{p}
=o(n^{-4}),
$$
and then by monotonicity of $|W_2^t|$,
$$
\probsub{G_{z_1, z_2}}{p}=\probsub{G_{z_1, z_2}\cap\left(\cup_{t\ge 0}\{|W_2^t|\ge 240(\log n)^2\}\right)}{p}
=o(n^{-4}).
$$
Thus for any fixed $z\in V$,
$$
\probsub{\text{more than
$1000(\log n)^2$ pairs $z\to z'$ are examined}}{p}
\le
\probsub{\bigcup_{z_2}G_{z, z_2}}{p}=o(n^{-4}),
$$
where the union is over four $z_2$ that are $\Gpuz$-neighbors of $z$.
This proves the first claim of the theorem with $C=1000$. The spatial and
temporal complexity bounds are then easy to deduce.
\end{proof}

\section*{Open problems}

We conclude with a list of open problems, some of which were mentioned in passing 
in the text, but most are introduced here.

\begin{mylist}
\item For the AE dynamics ($\tau=1$, $\theta=\infty$) on $\bZ_n^2$, can sharp transition be proved? Can one devise
a sequence of approximations (in principle 
computable in finite time) that provably converges to $p_c$?  
\item
For the dynamics with $\tau=1$ and $\theta=\infty$ on $\bZ_n^2$,
can one find a lower and an upper bound for $p_c$ of the form, respectively, 
$c_\sigma^\ell/\log n$ and $c_\sigma^u/\log n$, such that $c_\sigma^\ell\sim c_\sigma^u$
as $\sigma\to\infty$?
\item  When $\Gpuz$ is a regular tree with AE dynamics, 
can one show that  $\lim_{p\to 0} p\log \probsub{\grow}{p}$ exists 
and compute it?
\item To which precision can one estimate $p_c$ for AE dynamics on the hypercube 
or Hamming torus? (See Theorems~\ref{intro-lb-general} and~\ref{intro-ub-list}, and
Section 4 for the scaling results.)
\item How fast is the convergence to $\lambda_c$ in Theorems~\ref{intro-ring-thm} 
and~\ref{intro-theta2-thm}? (Bootstrap percolation 
is analyzed from this perspective in \cite{GHM}.)
\item Can the bounds in (\ref{intro-tau2-general}) be improved for $\sigma=1$? 
\item What is the scaling on three-dimensional torus with arbitrary $\sigma$,  
$\theta=2$ or $\theta=3$, and $\tau\le\theta$? Or on the $d$-dimensional torus, 
for general $d$? 
(Again, the answers are known for bootstrap percolation \cite{BBDM}.) 
\end{mylist}

\section*{Acknowledgements}
JG was partially supported by 
Republic of Slovenia's Ministry of Science program P1-285.  Thank you to Charlie Brummitt for his helpful comments on an earlier draft.

\end{document}